\documentclass{amsart}

\usepackage{amsmath}
\usepackage{amssymb}
\usepackage{amsthm}
\usepackage{mathrsfs}
\usepackage{graphicx}

\title[Invariants of ${\rm PSL}(n,\mathbb{R})$-Fuchsian representations]{Invariants of ${\rm PSL}(n,\mathbb{R})$-Fuchsian representations and a slice of Hitchin components}
\author{Yusuke Inagaki}
\address{Graduate~School~of~Science, Osaka~University}
\email{y-inagaki@cr.math.sci.osaka-u.ac.jp}
\begin{document}
\maketitle

\newtheorem{theorem}{Theorem}[section]
\newtheorem{corollary}[theorem]{Corollary}
\newtheorem{lemma}[theorem]{Lemma}
\newtheorem{proposition}[theorem]{Proposition}
\newtheorem{definition}[theorem]{Definition}
\newtheorem{question}[theorem]{Question}
\newtheorem{example}[theorem]{Example}
\newtheorem{remark}[theorem]{Remark}
\newtheorem{conjecture}[theorem]{Conjecture}
\newtheorem{caution}[theorem]{Caution}

\begin{abstract}
In this paper we show some properties of triangle invariants and shearing invariants of ${\rm PSL}(n,\mathbb{R})$-Fuchsian representations.
Moreover, using the Bonahon-Dreyer parameterization, we show that the Fuchsian locus of Hitchin components corresponds to a slice.
\end{abstract}

\section{Introduction}

In this paper we characterize ${\rm PSL}_n\mathbb{R}$-Fuchsian representations and verify how Teichm\"uller spaces are embedded into Hitchin components.
Let $S$ be a compact hyperbolic oriented surface.
The ${\rm PSL}_n\mathbb{R}$-Hitchin components $H_n(S)$ of $S$ is a prefered component of the ${\rm PSL}_n\mathbb{R}$-character variety ${\rm Hom}(\pi_1(S), {\rm PSL}_n\mathbb{R})/{\rm PSL}_n\mathbb{R}$.
Thee elements of $H_n(S)$ are called Hitchin representations.
These components are a higher dimensional analog of the Teichm\"uller space of $S$, which is the deformation space of hyperbolic structures of $S$.
The Hitchin component contains a subset $F_n(S)$ which corresponds to the Teichm\"uller space, called the Fuchsian locus.
The goal of this paper is to study the behavior of certain invariants of Hitchin representations on the Fuchsian locus, and to describe Fuchsian loci concretely.

For our purpose, we use the Bonahon-Dreyer parameterization of Hitchin components.
Let $\mathcal{L}$ be a maximal geodesic lamination of $S$ with finitely many leaves.
Such a lamination induces an ideal triangulation of $S$.
In particular we consider a maximal geodesic lamination associated to a pants decomposition of $S$, {\it i.e.} a maximal geodesic lamination whose closed leaves induces a pants decomposition of $S$.
For the lamination and the ideal triangulation, we can define three kinds of invariants of Hitchin representations:
(i) the triangle invariants for ideal triangles,
(ii) the shearing invariants for biinfinite leaves,
and (iii) the gluing invariants for closed leaves.
The Bonahon-Dreyer parameterization is defined by using these invariants.
This is a parameterization of $H_n(S)$ by the interior of a convex polytope in $\mathbb{R}^N$, where $N$ is a number depending on $\mathcal{L}$.
We denote the Bonahon-Dreyer parameterization associated to $\mathcal{L}$ by $\Phi_{\mathcal{L}} : H_n(S) \rightarrow \mathcal{P}_{\mathcal{L}} \subset \mathbb{R}^N$.
The main result of this paper is as follows.
\begin{theorem}
There is a slice $\mathcal{S}_{\mathcal{L}}$ of the range $\mathcal{P}_{\mathcal{L}}$ of the Bonahon-Dreyer parameterization associated to $\mathcal{L}$ such that the image $\Phi_{\mathcal{L}}(F_n(S))$ coincides with $\mathcal{S}_{\mathcal{L}}$.
\end{theorem}  

Moreover, we obtain the property of triangle, shearing, and gluing invariants on Fuchsian loci as corollary.
\begin{corollary}
A Hitchin representation is ${\rm PSL}_n\mathbb{R}$-Fuchsian if and only if
\begin{itemize}
\item[(i)] the triangle invariants are all zero, and
\item[(ii)] the shearing and gluing invariants are independent of their index.
\end{itemize}
\end{corollary}

\begin{remark}
Theorem 1.1 and Corollary 1.2 hold for any maximal geodesic laminations which consist of finitely many leaves.
We can use the argument of the proof of the main results to show this.
\end{remark}

\subsection*{Structure of this paper}
In Section 2, we recall the hyperbolic geometry of surfaces.
The tools, the shearing parameterization of a pair of pants, the Fenchel-Nielsen coordinate, the twist deformation, which are used in the proof of the main result, are defined in this section.
In Section 3, we define Hitchin components and recall properties of Hitchin representations, called the hyperconvex property and the Anosov property.
The Bonahon-Dreyer coordinate is defined in Section 4.
After the precise definition of the triangle, shearing, and gluing invariant, we recall the Bonahon-Dreyer parameterization theorem.
In Section 5, we show the only-if part of Corollary 1.2.
The proof is due to direct computations of the invariants.
In Section 6, we show the main result by using the technique of hyperbolic geometry of surfaces. 
Theorem 5.6 and Theorem 6.1 imply Theorem 1.1 and Corollary 1.2.
In Section 7, we refer to the case of surfaces with boundary.

\begin{remark}
The results of this paper are a generalization of \cite{I}.
We use a technique which is used in \cite{I} to show Proposition 5.1 and Proposition 5.5.  
\end{remark}

\subsection*{Acknowledgements}
The author would like to thank Shinpei Baba, Hideki Miyachi, and Ken'ichi Ohshika for their warm encouragement and valuable discussion.

\section{Hyperbolic geometry of surface}
\subsection{Hyperbolic structures of surfaces}
Let $S$ be a compact oriented surface.
We denote the hyperbolic plane of upper-half plane model by $\mathbb{H}^2$.
In this paper, we endow $\mathbb{H}^2$ with the orientation induced by the framing $<e_1, e_2>$, where $e_1 = (1,0)^t, e_2 = (0,1)^t$.
The group of orientation-preserving isometries ${\rm Isom}^+(\mathbb{H}^2)$ is isomorphic to the group ${\rm PSL}_2\mathbb{R}$, and the group ${\rm PSL}_2\mathbb{R}$ acts on $\mathbb{H}^2$ as linear fractional transformations.
A {\it hyperbolic metric} of $S$ is a complete Riemannian metric of constant curvature $-1$, which makes the boundary totally geodesic if $S$ has a nonempty boundary.
An isometric class of a hyperbolic metric on $S$ is often called a {\it hyperbolic structure} of $S$.
The hyperbolic structure of $S$ is related to a good representation of the fundamental group $\pi_1(S)$.
A representation $\rho : \pi_1(S) \rightarrow {\rm PSL}_2 \mathbb{R}$ is said to be {\it Fuchsian} if (i) $\rho$ is faithful and discrete,  and (ii) $\rho$ sends the boundary components to hyperbolic elements if $S$ has a nonempty boundary.
If $\rho : \pi_1(S) \rightarrow {\rm PSL}_2\mathbb{R}$ is Fuchsian, then there exists a subset $\Omega_{\rho}$, which is called a domain of discontinuity of $\rho$, such that $\rho(\pi_1(S))$ acts on $\Omega_{\rho}$ properly and $S_{\rho}= \rho(\pi_1(S)) \setminus \Omega_{\rho}$.  
The surface $S_{\rho}$ is a surface with a hyperbolic metric.
For a Fuchsian representation $\rho$, we can construct a $(\pi_1(S), \rho)$-equivariant local homeomorphism $f_{\rho} : \tilde{S} \rightarrow \mathbb{H}^2$ from the universal covering of $S$ to the hyperbolic plane. 
The image coincides with $\Omega_{\rho}$.
This map $f_{\rho}$ is called the {\it developing map} associated to $\rho$.
In this paper we assume that Fuchsian representations are orientation-preserving, i.e. the associated developing map is orientation-preserving.
In addition to, we suppose that the reference surface $S$ is given a hyperbolic metric.

\subsection{Geodesic laminations}
A {\it geodesic lamination} is a closed subset of $S$ which can be decomposed to a disjoint union of simple complete geodesics called {\it leaves}. 
Geodesic laminations consist of closed and biinfinite geodesics, and we call them {\it closed leaves} and {\it biinfinite leaves} respectively. 
The concept of geodesics depends on a hyperbolic metric of $S$. 
We remark that there exists a natural bijection between the set of $g_1$-geodesic laminations and the set of $g_2$-geodesic laminations for different hyperbolic metrics $g_1$ and $g_2$ of $S$.
In particular, for any hyperbolic metric $g$ and any simple curve $c$ on $S$, there is a $g$-geodesic $c_g$ which is isotopic to $c$. 
A geodesic lamination is said to be {\it maximal} if it is properly contained in no other geodesic lamination.
In this paper, we consider only laminations consisting of finitely many leaves.
For a geodesic lamination $\mathcal{L}$ of $S$, the preimage $\tilde{\mathcal{L}}$ of $\mathcal{L}$ in $\tilde{S}$ gives a geodesic lamination of $\mathbb{H}^2$.
A connected component of the closure of $\mathbb{H}^2 \setminus \tilde{\mathcal{L}}$ is called a {\it plaque}.
A geodesic lamination is {\it oriented} if each leaf is oriented.
We may choose the orientation of each leaf independently.
Given maximal geodesic lamination $\mathcal{L}$, we define a short arc system for closed leaves as an additional data.
A {\it short arc system} $K = \{ K_C \}_C$ is a family of an arc $K_C$ defined for each closed leaf $C$ of $\mathcal{L}$ which satisfies two condition (i),(ii) below: 
(i) The arc $K_C$ is transverse to $\mathcal{L}$ and the intersection $K_C \cap C$ is just one point $x$. 
(ii) Let $K_1$ and $K_2$ be component of $K_C \setminus C$.
Then there exists an immersion $f_i : K_i \times [0, +\infty) \rightarrow S$ such that $f_i(x,0) = x$, the subset $\{x \} \times [0, \infty)$ parametrizes a geodesic with unit speed spiraling along $C$, and the image $f_i(x,[0, \infty))$ is contained in a leaf of $\mathcal{L}$ if $x \in \mathcal{L} \cap K_i$.
We denote, by $\mathcal{L}_K$, the geodesic lamination $\mathcal{L}$ with a short arc system $K$.
Note that the closed leaf which intersect to $K_C$ is only $C$ by definition.

\subsection{Teichm\"uller space}

The {\it Teichm\"uller space} $\mathscr{T}(S)$ of $S$ is defined by
\[ \mathscr{T}(S) = \{ \rho : \pi_1(S) \rightarrow {\rm PSL}_2\mathbb{R} ~|~ \mbox{Fuchsian} \}/{\rm PSL}_2\mathbb{R} \]
where the quotient is defined by the conjugate action of ${\rm PSL}_2\mathbb{R}$ on the set of Fuchsian representations.
The topology of $\mathscr{T}(S)$ is given by the compact open topology.
The Teichm\"uller space is the deformation space of hyperbolic structures of $S$.
Let ${\rm Hyp}(S)$ be the set of hyperbolic metrics of $S$, and ${\rm Diff}_0(S)$ be the identity component of the group of diffeomorpshisms of $S$.
The group ${\rm Diff}_0(S)$  acts on ${\rm Hyp}(S)$ by pull-back.
The Teichm\"uller space is also defined by $\mathscr{T}(S) ={\rm Hyp}(S) / {\rm Diff}_0(S)$.
This definition is equivalent to the definition by Fuchsian representations.
If we have a Fuchsian representation of $\pi_1(S)$, then the associated hyperbolic metric is defined by the covering $\Omega_{\rho} \rightarrow S$.
Conversely, for any a hyperbolic metric $g$ of $S$, there is an orientation-preserving isometric embedding $f_g : \tilde{S}_g \rightarrow \mathbb{H}^2$ where $\tilde{S}_g$ is the universal covering of $S_g$ with pullback of $g$. 
Then we can take a representation $\rho : \pi_1(S) \rightarrow {\rm PSL}_2\mathbb{R}$ such that $f_g$ is $(\pi_1(S), \rho)$-equivariant.
This representation is Fuchsian. 
There are some equivalent definitions of $\mathscr{T}(S)$, see \cite{IT}.

\subsection{Parameterizations of hyperbolic structures of a pair of pants}
\subsection*{Length parameterization}
We consider some parameterizations of the Teichm\"uller space of a pair of pants.
Note that a pair of pants is oriented.
It is well known that hyperbolic structures of a pair of pants $P$ is uniquely determined by the length of boundary components.
\begin{theorem}(\cite{IT}, Theorem 3.5.)
Let $C_1, C_2, C_3$ be boundary components of $P$.
Then the map 
\[ \mathscr{T}(P) \rightarrow \mathbb{R}_{>0}^3 : \rho \mapsto (l_{\rho}(C_1), l_{\rho}(C_2), l_{\rho}(C_3)) \]
is a diffeomorphism, where $l_{\rho}(\cdot)$ is the length function associated to a hyperbolic structure $\rho$. 
\end{theorem}

\subsection*{Shearing parameterization}
We give another parameterization of $\mathscr{T}(P)$ by the shearing parameter along ideal triangles.
An ideal triangle is a geodesic triangle in $\mathbb{H}^2$ which has vertices at infinite boundary.
This is unique up to isometry.
Let us consider two ideal triangles $\triangle(x, y, z_0), \triangle(x,y,z_1)$ which are adjacent along the side $[x,y]$.
For each triangle, we draw the geodesic $p_0, p_1$ joining $z_0$, $z_1$ to $[x,y]$ which is orthogonal to $[x,y]$.
Let $b_i = p_i \cap [x,y]$.
The {\it shearing parameter} $\sigma(\triangle(x, y, z_0), \triangle(x,y,z_1) )$ of $\triangle(x,y, z_0)$ and $\triangle(x,y,z_1)$ along $[x,y]$ is a signed hyperbolic distance $d(b_0, b_1)$.
If $b_1$ is on the left side of $b_0$ with respect to the direction of $[z_0, b_0]$ from $z_0$ to $b_0$, then we define the sign of  $\sigma(\triangle(x, y, z_0), \triangle(x,y,z_1) )$ is positive.
See Figure 1.
\begin{figure}[htbp]
\begin{center}
\includegraphics[width = 7cm, height=4cm]{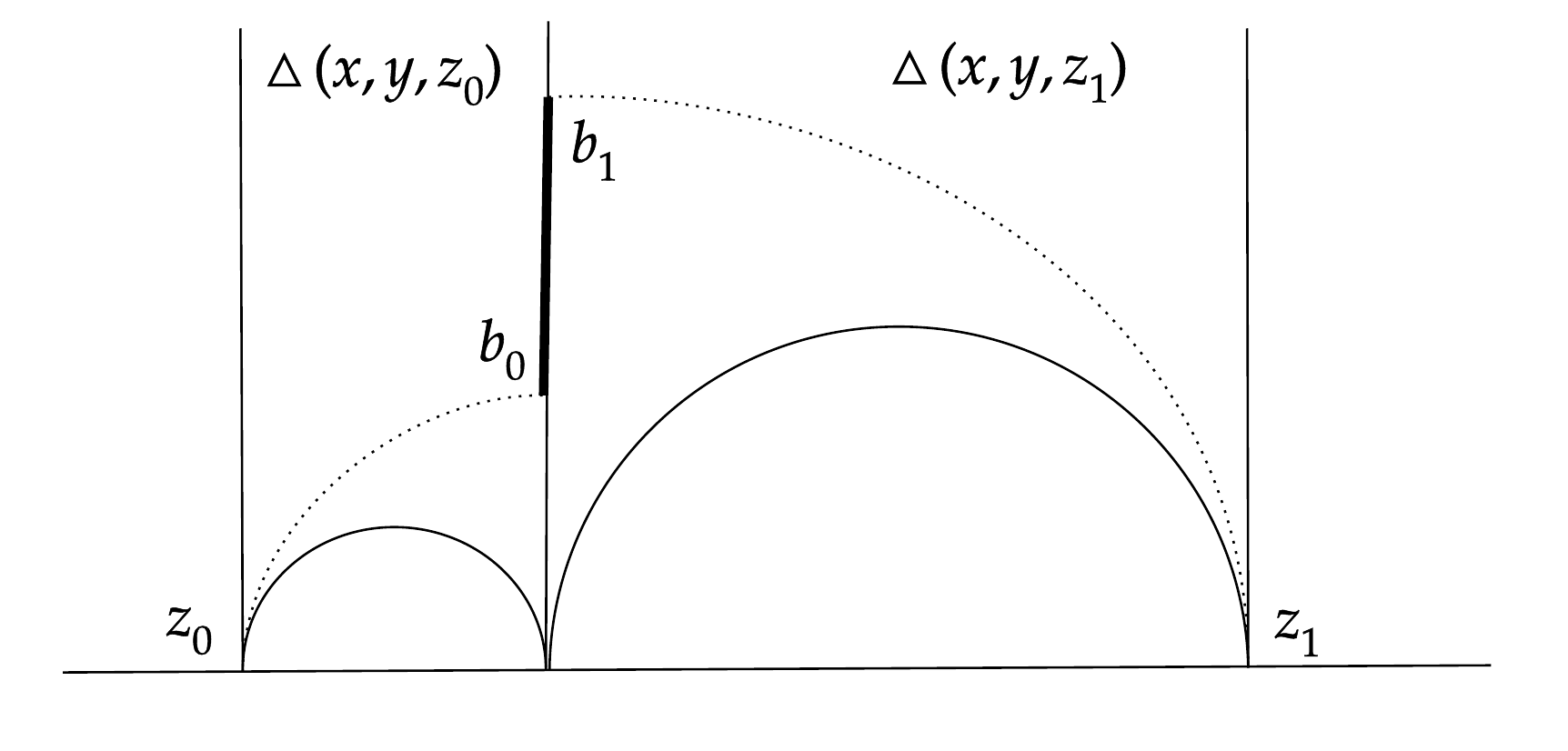}
\caption{The shearing parameter is positive.}
\end{center}
\end{figure}
We can describe shearing parameters in terms of cross ratio.
 
\begin{definition}
Let $a,b,c,d \in \partial \mathbb{H}^2$ be a quadruple of distinct points of the ideal boundary $\partial \mathbb{H}^2$.
The cross ratio $z(a,b,c,d)$ is the ratio
\[ z(a,b,c,d) = \dfrac{(d-a)(b-c)}{(d-c)(b-a)}. \]
\end{definition}
\begin{remark}
The cross ratio  of $a,b,c,d \in \partial \mathbb{H}^2$ is usually defined by
\[ z'(a,b,c,d) = \dfrac{(a-c)(b-d)}{(a-d)(b-c)}. \]
Two definitions have the relation $z(a,b,c,d) = z'(d,b,a,c)$.
The prefered point  of our definition is to satisfy $ z(0,1,\infty,d) = d$. 
\end{remark}

Let $B$ be a biinfinite leaf with the end points $x,y$.
We consider two ideal triangles $T^l = \triangle (x,z^l,y)$ and $T^r = \triangle(x, y, z^r)$ where the points $x,z^l,y,z^r$ are in counterclockwise order.
The following relation is given by a direct computation.
\begin{proposition}The following relation holds.
\[ \sigma(T^l, T^r) = \log -z(y, z^r, x , z^l)^{-1}. \]
\end{proposition}

Using shearing parameters, we can parameterize hyperbolic structures of a pair of pants $P$.
Consider a maximal geodesic lamination of $P$.
Maximal geodesic laminations of $P$ which consist of finitely many leaves are classified into 2 types (I) and (II) as in Figure 2 and Figure 3.
\begin{figure}[htbp]
\begin{minipage}{0.45\hsize}
\begin{center}
\includegraphics[width = 40mm, height=2cm]{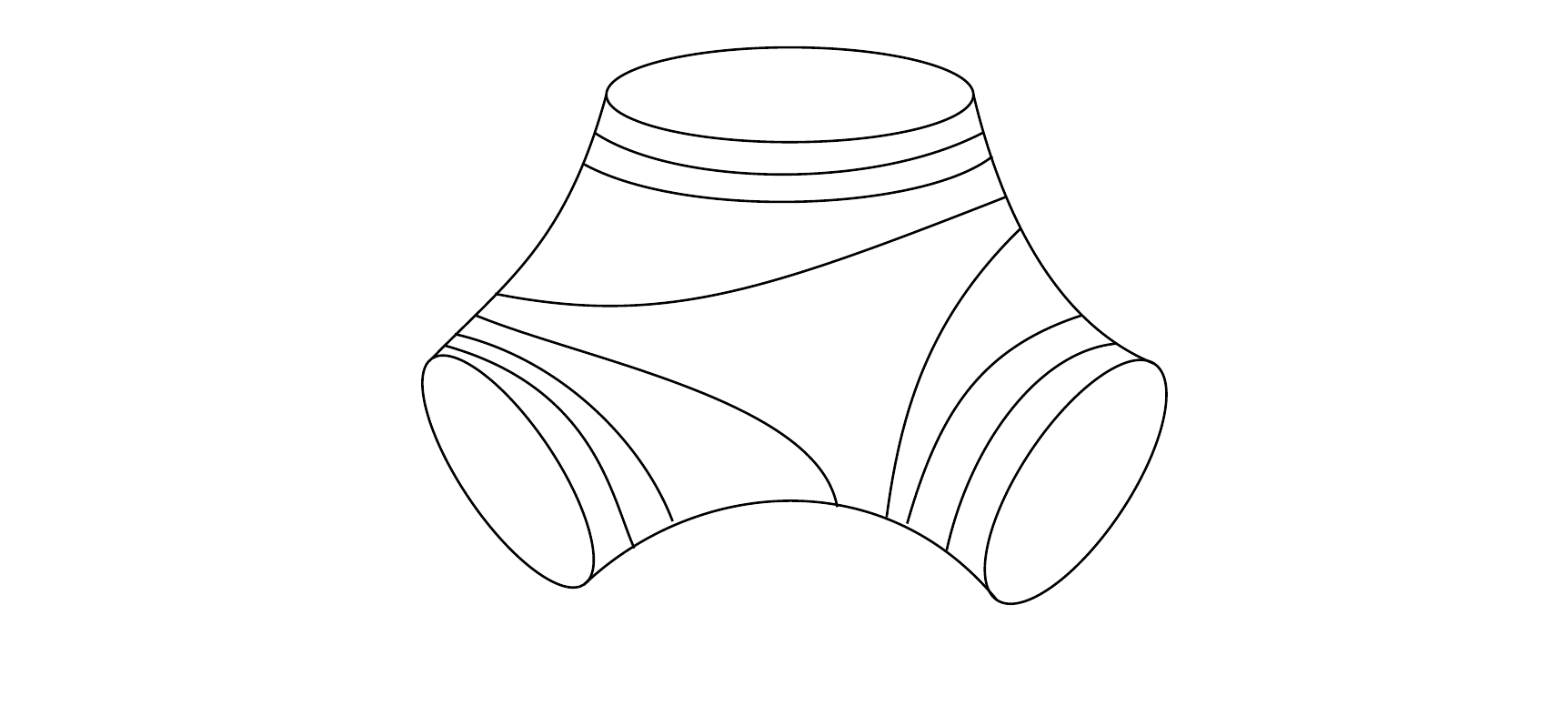}
\caption{A lamination of type 1.}
\end{center}
\end{minipage}
\begin{minipage}{0.45\hsize}
\begin{center}
\includegraphics[width=40mm, height=2cm]{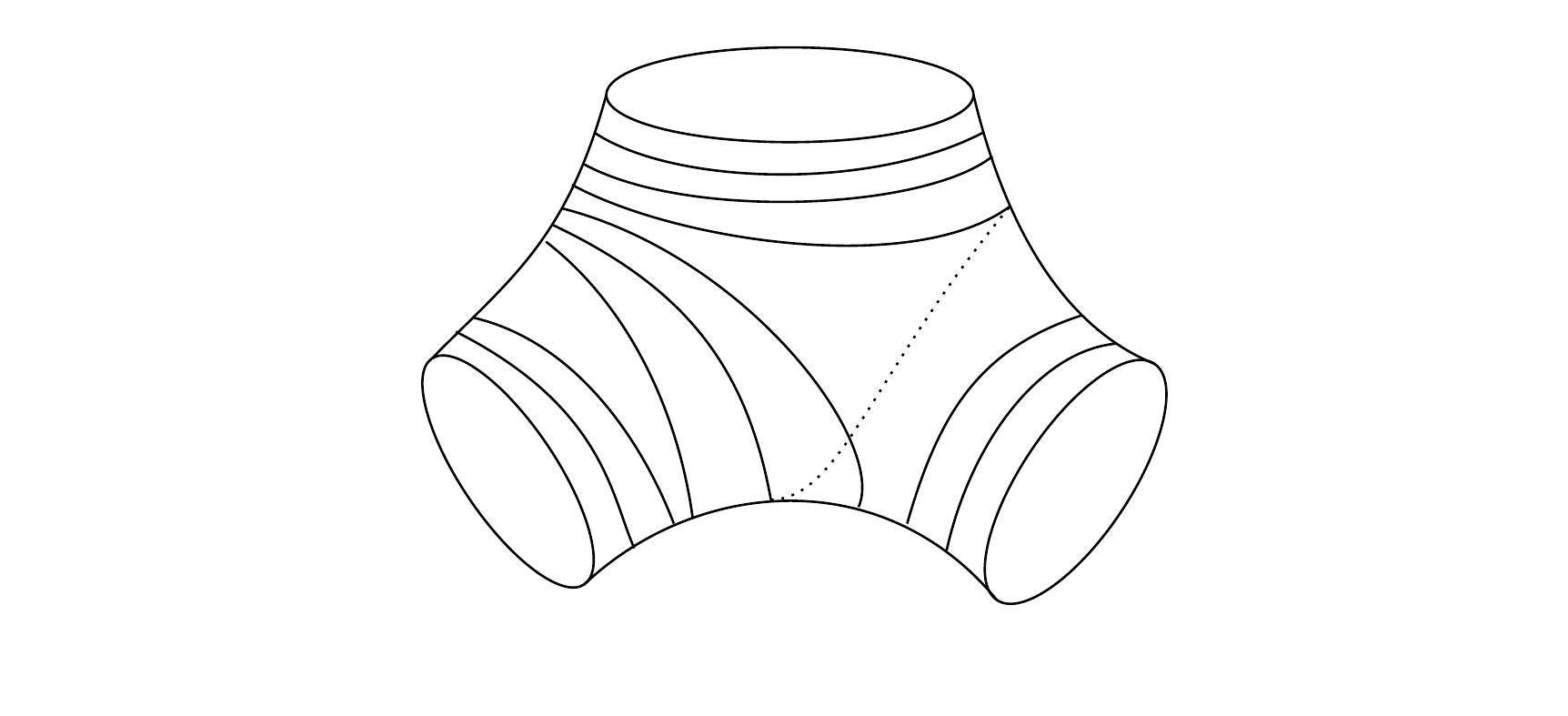}
\caption{A lamination of type 2.}
\end{center}
\end{minipage}
\end{figure}
The lamination of type (I) is represented by $\{ C_1, C_2, C_3, B_{12}, B_{23}, B_{31} \}$, where $C_i$ is a boundary component and $B_{ij}$ is a biinfinite leaf which spirals to $C_i$ and $C_j$.
The lamination of type (II) is represented by $\{ C_1, C_2, C_3, B_{ii}, B_{ij}, B_{ik} \}$.
They contain a biinfinite leaf spiraling to the same closed leaf in its ends.
Moreover we characterize these laminations by the direction of the spiraling.
When the spiraling occurs in the direction opposite to the orientation of pants, we call the spiraling {\it positive spiraling}.
See Figure 4.
Similarly, we call the spiraling in Figure 5 {\it negative spiraling}.
Maximal geodesic laminations on $P$ are classified by types and signatures of the spiraling.

\begin{figure}[htbp]
\begin{minipage}{0.45\hsize}
\begin{center}
\includegraphics[width = 40mm, height=3cm]{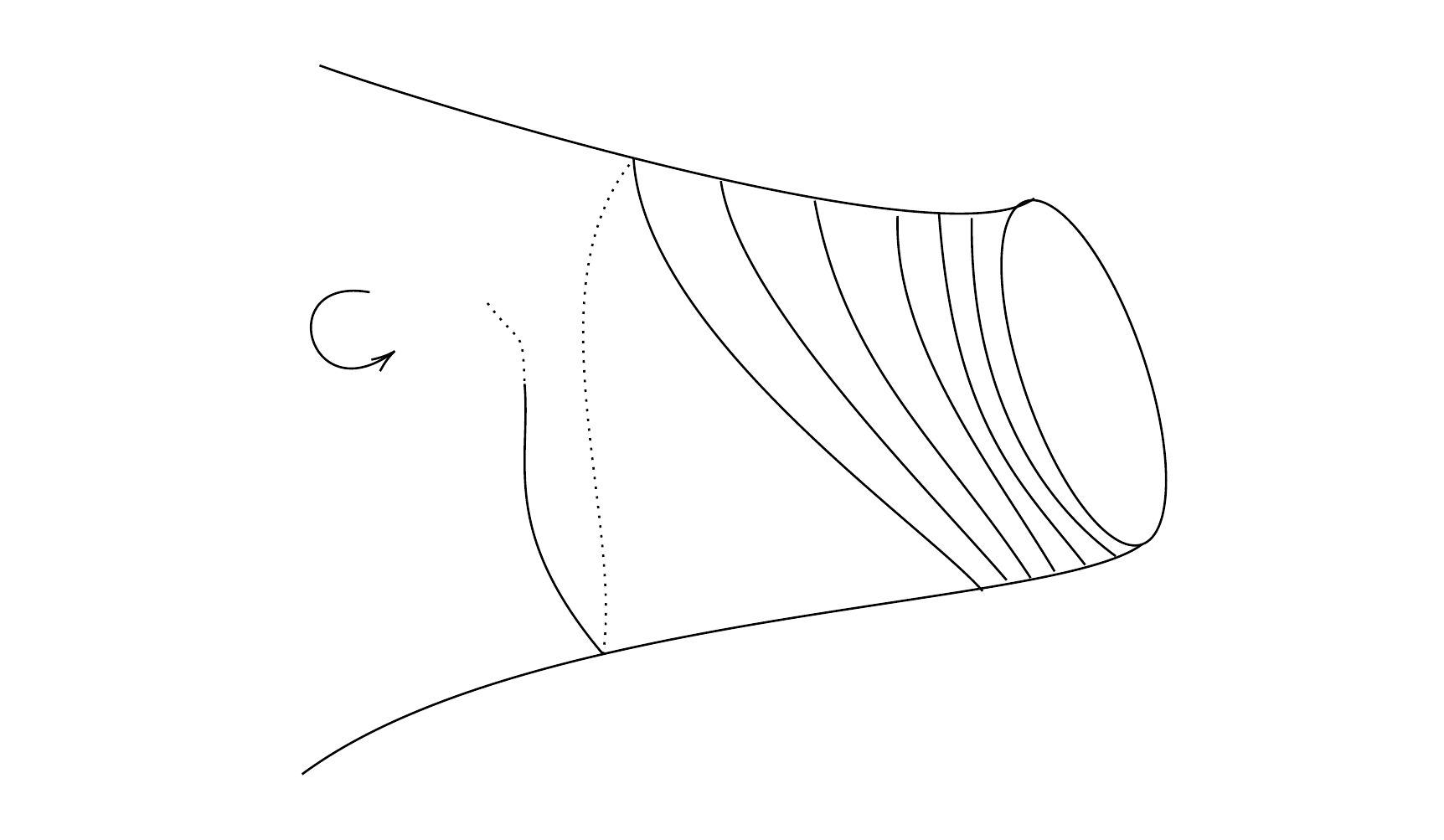}
\caption{Positive spiraling.}
\end{center}
\end{minipage}
\begin{minipage}{0.45\hsize}
\begin{center}
\includegraphics[width=40mm, height=3cm]{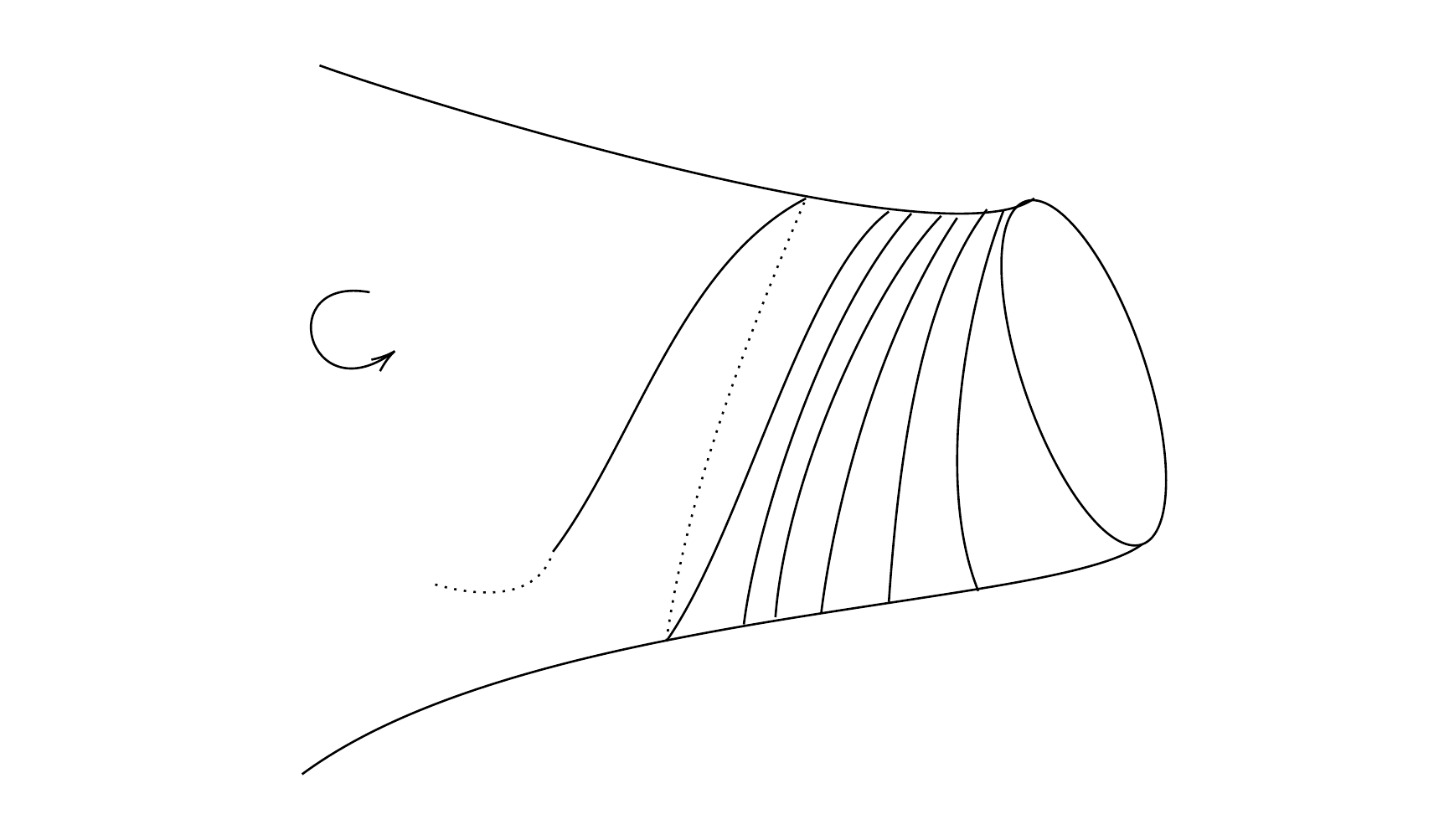}
\caption{Negative spiraling.}
\end{center}
\end{minipage}
\end{figure}

We fix a maximal geodesic lamination $\mathcal{L} = \{ C_1, C_2, C_3, B_1, B_2, B_3\}$ of $P$.
Note that both types (I) and (II) consist of three closed leaves and three biinfinite leaves.
This lamination induces an ideal triangulation of $P$.
Let $\rho \in \mathscr{T}(P)$ be a hyperbolic structure and $f_{\rho} : \tilde{P} \rightarrow \mathbb{H}^2$ be the associated developing map.
The shearing parameter of $\rho$ along $B_i$ is defined as follows.
Lift $B_i$ to $\tilde{B}_i$, which is a biinfinite geodesic in the universal covering $\tilde{P}$.
We denote the end points of $\tilde{B}_i$ by $x$ and $y$.
Under the ideal triangulation, $\tilde{B}_i$ is adjacent to two ideal triangles $T^l = \triangle (x,y,z^l)$ and $T^r = \triangle (x,y, z^r)$.
Here the vertices $z^r$ and $z^l$ are determined so that $x,z^l, y, z^r$ are in counterclockwise order.
We define the shearing parameter $\sigma^{\rho}(B_i)$ by the shearing parameter $\sigma(f_{\rho}(T^l), f_{\rho}(T^r))$.

\begin{proposition}(See \cite{Ma}, Proposition 7.4.9.)
There is an analytic embedding
\[ \sigma_{\mathcal{L}} ~:~\mathscr{T}(P) \rightarrow \mathbb{R}^3 : \rho \mapsto (\sigma^{\rho}(B_1), \sigma^{\rho}(B_2), \sigma^{\rho}(B_3)). \]

\end{proposition}

To describe the range of this parameterization, we consider the relation between the shearing parameter and the boundary length, both of which determine hyperbolic structures of $P$.
For a closed leaf $C_i$, we suppose that biinfinite leaves $B_1, \cdots, B_k$ spiral to $C_i$.
\begin{proposition}[\cite{Ma}, Proposition 7.4.8]
If the spiraling of $B_i$ is positive, then
\[  l_{\rho}(C_i) = \sum_{j=1}^k \sigma^{\rho}(B_j), \]
and if the spiraling of $B_i$ is negative, then 
\[ l_{\rho}(C_i) = - \sum_{j=1}^k \sigma^{\rho}(B_j). \]
\end{proposition}
Consider a maximal geodesic lamination of type (I).
When we represent the lamination by $\mathcal{L} = \{ C_1, C_2, C_3, B_{12}, B_{23}, B_{31} \}$,  the shearing parameterization associated to $\mathcal{L}$ is defined by $\sigma_{\mathcal{L}}(\rho) = (\sigma^{\rho}(B_{12}), \sigma^{\rho}(B_{23}), \sigma^{\rho}(B_{31}))$.
The range of this parameterization is described as follows;
\[ \{ (x_{12}, x_{23}, x_{31}) \in \mathbb{R}^3 ~|~\forall i,j,k ~\mbox{with}~ \{i,j,k\} = \{1,2,3\}, {\rm sgn}(C_i) ( x_{ij} + x_{ik} ) > 0 \}, \]
where sgn($C_i$) is the signature of spiraling along $C_i$.
In the case of laminations of typer (II), we consider $\mathcal{L} = \{ C_1, C_2, C_3, B_{ii}, B_{ij}, B_{ik} \}$ and the associated shearing parameterization $\sigma_{\mathcal{L}}(\rho) = (\sigma^{\rho}(B_{ii}), \sigma^{\rho}(B_{ij}), \sigma^{\rho}(B_{ik}))$.
The range of this parameterization is equal to the following space;
\[ \{ (x_{ii}, x_{ij}, x_{ik}) \in \mathbb{R}^3 ~|~x_{ij} > 0, x_{ik} > 0, {\rm sgn}(C_i)(x_{ii}+x_{ij} + x_{ik}) > 0 \}. \]

\subsection{Fenchel-Nielsen coordinate and twist deformations}
In this subsection, we recall the Fenchel-Nielsen coordinate, which is a global coordinate of $\mathscr{T}(S)$.
See the detail in Section 3.2 of \cite{IT}.  
To define this coordinate, we recall a pants decomposition of surfaces.
It is known that any compact orientable surface $S$ of negative Euler characteristic number $\chi(S)$ with $b$ boundary components is decomposed into $|\chi(S)|$ pairs of pants by a family of $\frac{3|\chi(S)| - b}{2}$ disjoint simple closed curves .
If $S$ is decomposed into pairs of pants $\mathcal{P} = \{ P_1, \cdots, P_{|\chi(S)|} \}$ along simple closed curves $\mathcal{C} = \{ C_1, \cdots, C_{\frac{3|\chi(S)| - b}{2}}\}$, we call $\mathcal{P}$ a pants decomposition of $S$ and $C_i \in \mathcal{C}$ the decomposing curves of $\mathcal{P}$.
We suppose that $C_i$ is geodesic. 

The Fenchel-Nielsen coordinate is a coordinate of $\mathscr{T}(S)$ by the hyperbolic length of decomposing curves and the twist parameter along decomposing curves.
We define the twist parameter.
Suppose two pairs of pants $P_1$ and $P_2$ are glued along the closed geodesic $C$ which is a boundary component of $P_1$ and $P_2$.
We fix a hyperbolic structure $\rho$ of $P_1 \cup_C P_2$.
For each pants, we fix an orthogonal arc $H_i$ which joins $C$ and an other boundary component of $P_i$.
Such an arc exists since there is an isometric involution of a pair of pants, and its fixed set consists of three geodesics which are orthogonal to two boundary components.
One can choose this geodesic as an orthogonal arc.
To define the twist parameter, we lift $P_i$ to $\tilde{P}_i$ which is a subset of the universal covering of $P_1 \cup_C P_2$ so that $\tilde{P}_i$ are adjacent.
Take lifts $\tilde{C}$ and $\tilde{H}_i$ of the arcs $C$ and $H_i$ so that they are on $\tilde{P}_i$.
Then the {\it twist parameter} $\theta^{\rho}(C)$ along $C$ is defined  by 
\[ \theta^{\rho}(C) = 2 \pi \dfrac{{\rm Length}_{\rho}(H_1, H_2)}{{\rm Length}_{\rho}(C)} \]
where ${\rm Length}_{\rho}(C)$ is the $\rho$-length of the closed curve $C$, and ${\rm Length}_{\rho}(H_1, H_2)$ is the signed $\rho$-length between the end points $\tilde{H}_1 \cap \tilde{C}$ and $\tilde{H}_2 \cap \tilde{C}$.
The signature of ${\rm Length}_{\rho}(H_1, H_2)$ is positive if $\tilde{H}_1$ and $\tilde{H}_2$ are as in Figure 6.

\begin{figure}[htbp]
\begin{center}
\includegraphics[width = 7cm, height=4cm]{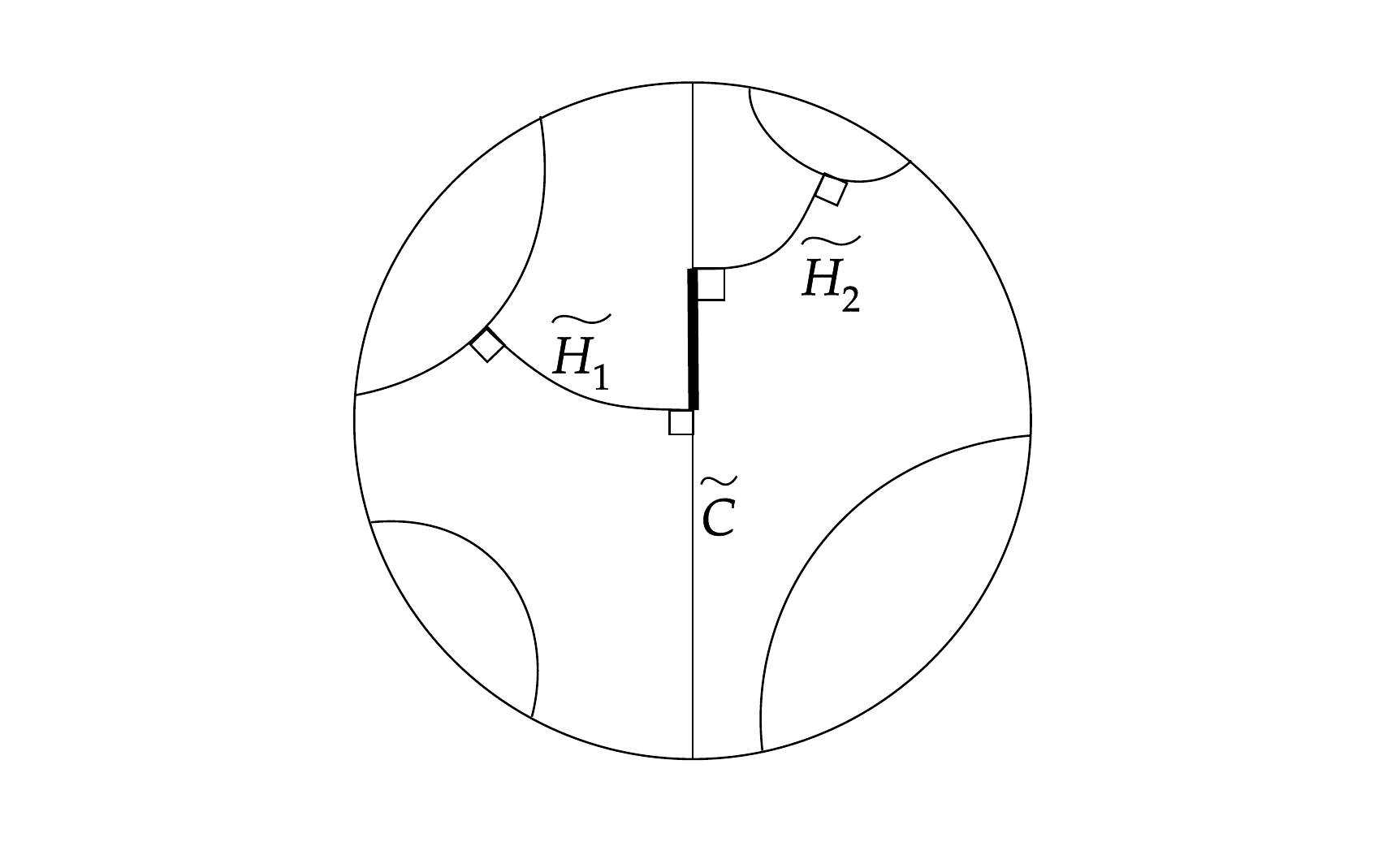}
\caption{The twist parameter is positive.}
\end{center}
\end{figure}

Fix a hyperbolic structure $\rho$ of $S$ and a pants decomposition of $S$.
We denote boundary components of $S$ by $\partial_1, \cdots, \partial_b$.
For a pants decomposition of $S$ by $\{ C_i \}$, the {\it Fenchel-Nielsen coordinate} is defined by
\[ FN : \mathscr{T}(S) \rightarrow \mathbb{R}^{3|\chi(S)|} : \rho \mapsto (l_{\rho}(C_i), \cdots, l_{\rho}(\partial_i), \cdots,  \theta^{\rho}(C_i), \cdots ). \]

We recall the twist deformation of hyperbolic structures which corresponds to the change of twist parameters.
Let $C$ be a decomposing curve of a pants decomposition of $S$ and $\rho$ be a hyperbolic structure of $S$.
We take the preimage $\mathscr{C}$ of $C$ by the covering $f_{\rho} (\tilde{S}) \rightarrow S_{\rho}$, which is a geodesic lamination of $\mathbb{H}^2$.
Choose a leaf $\tilde{C} \in \mathscr{C}$.
The geodesic $\tilde{C}$ is a side of two plaque $Q^l$ and $Q^r$ of $\mathscr{C}$.
We orient $\tilde{C}$ so that the plaque $Q^l$ is on the left of $\tilde{C}$ with respect to the orientation.
Let ${\rm tw}_t$ be a hyperbolic isometry with the axis $\tilde{C}$, which is conjugate to 
\[
\begin{bmatrix}
\exp(t) && 0 \\
0 && \exp(-t) \\
\end{bmatrix}
\]
by the normalization which sends the attracting (resp. repelling) point of $\tilde{C}$ to $\infty$ (resp. 0).
Glue ${\rm tw}_t(Q^l)$ and $Q^r$ along $\tilde{C}$.
Iterating this operation for all leaves of $\mathscr{C}$, we obtain a new developing map ${\rm Tw}_t \circ f_{\rho} : \tilde{S} \rightarrow \mathbb{H}^2$ where ${\rm Tw}_t$ is a map $\mathbb{H}^2 \rightarrow \mathbb{H}^2$ induced by the iteration.
The developing map induces an element $\eta$ of $\mathscr{T}(S)$.
We call $\eta$ a {\it twist deformation} of $\rho$.

\section{Hitchin representations and their properties}

\subsection{Hitchin components}

Let $\Gamma$ be a finitely generated group.
The ${\rm PSL}_n\mathbb{R}$-{\it representation variety} $\mathcal{R}_n(\Gamma)$ of $\Gamma$ is the set of group homomorphisms $\mathcal{R}_n(\Gamma)={\rm Hom}(\Gamma, {\rm PSL}_n(\mathbb{R}))$ with the compact open topology.
${\rm PSL}_n\mathbb{R}$ acts on the representation variety by conjugation.
The quotient space $\mathcal{X}_n(\Gamma)=\mathcal{R}_n(\Gamma)/ {\rm PSL}_n(\mathbb{R})$ is called the ${\rm PSL}_n(\mathbb{R})$-{\it character variety}. 
When the finitely generated group $\Gamma$ is the fundamental group of a manifold $M$, we denote the representation (resp. character) variety $\mathcal{R}_n(\pi_1(M))$ (resp. $\mathcal{X}_n(\pi_1(M))$) by $\mathcal{R}_n(M)$ (resp. $\mathcal{X}_n(M)$) simply.
When $\Gamma = \pi_1(S)$, then the Teichm\"uller space $\mathscr{T}(S)$ is naturally embedded in the character variety $\mathcal{X}_2(S)$ by definition.
It is known that $\mathscr{T}(S)$ is a connected component of $\mathcal{X}_2(S)$.
(See \cite{Go88}.)

The Hitchin component is a perfered component of $\mathcal{X}_n(S)$ which contains $\mathscr{T}(S)$.
Let us consider an irreducible representation ${\rm SL}_2\mathbb{R} \rightarrow {\rm SL}_n\mathbb{R}$ which is unique up to equivalence.
This representation is obtained by the symmetric power.
We denote its projectivization ${\rm PSL}_2\mathbb{R} \rightarrow {\rm PSL}_n\mathbb{R}$ by $\iota_n$.
The representation $\iota_n$ induces a map between character varieties $(\iota_n)_* : \mathcal{X}_2(S) \rightarrow \mathcal{X}_n(S)$ by the correspondence $\rho \mapsto \iota_n \circ \rho$.
Since $\iota_n$ is a group homomorphism, this induced map is well-defined.
When $\partial S = \emptyset$, the Hitchin component is defined as below.
\begin{definition}
The $({\rm PSL}_n\mathbb{R}-)$ Hitchin component $H_n(S)$ is the connected component of $\mathcal{X}_n(S)$ which contains the image $F_n(S) = (\iota_n)_*(\mathscr{T}(S))$.
\end{definition}
When $\partial S \neq \emptyset$, a representation $\rho : \pi_1(S) \rightarrow {\rm PSL}_n\mathbb{R}$ is said to be {\it purely loxodromic} if the image of boundary components via $\rho$ is conjugate to an element in the interior of a Weyl chamber, so an element with distinct, only real eigenvalues.
We denote the space of purely-loxodromic representations by $\mathcal{R}_n^{loxo}(S)$, and $\mathcal{X}_n^{loxo}(S) = \mathcal{R}_n^{loxo}(S) / {\rm PSL}_n\mathbb{R}$.
Note that $(\iota_n)_*(\mathscr{T}(S))$ consists of only purely loxodromic elements.
The (${\rm PSL}_n\mathbb{R}$-) Hitchin components $H_n(S)$ is the connected component of $\mathcal{X}^{loxo}_n(S)$ which contains the image $F_n(S)=(\iota_n)_*(\mathscr{T}(S))$.

We call the image $F_n(S)$ of $\mathscr{T}(S)$ the {\it Fuchsian locus} of $H_n(S)$.
{\it Hitchin representations} are representations $\rho : \pi_1(S) \rightarrow {\rm PSL}_n\mathbb{R}$ whose conjugacy class belongs to $H_n(S)$.
A Hitchin representation $\rho$ is ${\rm PSL}_n\mathbb{R}$-{\it Fuchsian} if $\rho$ is contained in $F_n(S)$, {\it i.e.} there is a Fuchsian representation $\rho_0 : \pi_1(S) \rightarrow {\rm PSL}_2\mathbb{R}$ such that $\rho = \iota_n \circ \rho_0$. 

We remark the homeomorphic type of Hitchin components of closed surfaces.

\begin{theorem}[Hitchin \cite{Hi92} ]
If the surface $S$ is closed, the Hitchin component $H_n(S)$ is homeomorphic to $\mathbb{R}^{(2g-2)(n^2-1)}$.

\end{theorem}

\begin{remark}
If it is clear from context, we call ${\rm PSL}_n\mathbb{R}$-Fuchsian representations Fuchsian representations simply.
In addition to, if we confuse Fuchsians representations which are elements of the Teichm\"uller space,  and ${\rm PSL}_n\mathbb{R}$-Fuchsian representations, we call Fuchsian representations hyperbolic holonomy.
\end{remark}
\begin{caution}
In the following, we consider only closed surfaces.
Non-closed case is discussed in Section 7.
\end{caution}
\subsection{Hyperconvex property}
The projective special linear group ${\rm PSL}_n\mathbb{R}$ acts on the projective space $\mathbb{RP}^{n-1} = P(\mathbb{R}^n)$ by the projectivization of linear action of ${\rm SL}_n\mathbb{R}$ on $\mathbb{R}^n$.
We define the hyperconvexity of projective linear representations of $\pi_1(S)$.
Let $\partial \pi_1(S)$ be the ideal boundary of $\pi_1(S)$ which is a visual boundary of a Cayley graph of $\pi_1(S)$.
Note that $\partial \pi_1(S)$ is homeomorphic to $\partial \tilde{S}$ through a hyperbolic structure of $S$.
Therefore, in this paper, we identify $\partial \pi_1(S)$ with $\partial \tilde{S}$ by using the reference hyperbolic structure of $S$.
\begin{definition}
A representation $\rho : \pi_1(S) \rightarrow {\rm PSL}_n\mathbb{R}$ is said to be hyperconvex if there exists a $(\pi_1(S), \rho)$-equivariant continuous map $\xi_{\rho}:\partial \pi_1(S) \rightarrow \mathbb{RP}^{n-1}$ such that $\xi_{\rho}(x_1) + \cdots + \xi_{\rho}(x_n)$ is direct for any pairwise distinct points $x_1, \cdots, x_n \in \partial \pi_1(S)$.
\end{definition}
The associated curve $\xi_{\rho}$ is called the {\it hyperconvex curve} of $\rho$.
All Hitchin representations have hyperconvex property.
Labourie showed that Hitchin representations are hyperconvex by Anosov property which is explained in the next subsection.
Moreover the converse result was shown by Guichard.
Hence the following result holds.
\begin{theorem}[Guichard \cite{Gu08}, Labourie \cite{La06}]
A representation $\rho : \pi_1(S) \rightarrow {\rm PSL}_n\mathbb{R}$ is Hitchin if and only if $\rho$ is hyperconvex.

\end{theorem}

Moreover Labourie showed the following.
\begin{theorem}[\cite{La06}]
Let $\rho :\pi_1(S) \rightarrow {\rm PSL}_n\mathbb{R}$ be a hyperconvex representation with the hyperconvex curve $\xi_{\rho} : \partial \pi_1(S) \rightarrow \mathbb{RP}^{n-1}$.
Then there exists a unique curve $\xi_{\rho}^i : \partial \pi_1(S) \rightarrow {\rm Gr}^k(\mathbb{R}^n)$ with the properties from (i) to (iv) below.
\begin{itemize}
\item[(i)] $\xi^p(x) \subset \xi^{p+1}(x)$ for any $x \in \partial \pi_1(S) $.
\item[(ii)] $\xi^1(x) = \xi_{\rho}(x) $ for any $x \in \partial \pi_1(S) $.
\item[(iii)] If $n_1, \cdots , n_l$ are positive integers such that $ \sum n_i \leq n$, then $\xi^{n_1}(x_1) + \cdots + \xi^{n_l}(x_l)$ is direct for any pairwise distinct points $x_1, \cdots, x_l \in \partial \pi_1(S)$.
\item[(iv)] If $n_1, \cdots , n_l$ are positive integers such that $ p =\sum n_i \leq n$, then
\[ 
\lim_{(y_1, \cdots, y_l) \rightarrow x; y_i \mbox{distinct}}  \xi^{n_1}(y_1) + \cdots + \xi^{n_l}(y_l) \rightarrow \xi^p(x)
\]
\end{itemize} 

\end{theorem} 
Theorem 3.7 implies that any hyperconvex curves are extended to curves into the flag manifold.
(See Section 4.1 for the precise definition of flags.)
The map $(\xi^1, \cdots, \xi^{n-1}) : \partial \pi_1(S) \rightarrow {\rm Flag}(\mathbb{R}^n)$ is called the {\it (osculating) flag curve} of the hyperconvex curve $\xi_{\rho}$.

We can explicitly describe the hyperconvex curve of Fuchsian representations.
Let $\rho_n = \iota_n \circ \rho$ be a Fuchsian representation.
Recall that the irreducible representation $\iota_n$ is defined by symmetric power of the representation $({\rm SL}_2\mathbb{R}, \mathbb{R}^2)$.
We identify $\mathbb{R}^n$ with $Sym^{n-1}(\mathbb{R}^2)$.
Consider the Veronese embedding $\nu : \mathbb{RP}^1 \rightarrow \mathbb{RP}^{n-1}$ defined by sending $[a: b]$ to $[a^{n-1} : a^{n-2}b : \cdots : b^{n-1}]$.
Then the composition $\nu \circ f_{\rho}$ of the Veronese embedding with the developing map gives the hyperconvex curve of $\rho_n$.
Using homogeneous polynomials, the flag is also described explicitly.
The symmetric power $Sym^{n-1}(\mathbb{R}^2)$, which is identified with $\mathbb{R}^n$, is also identified with the vector space 
\[ {\rm Poly}_n(X,Y) = \{ a_1 X^{n-1} + a_2 X^{n-2}Y + \cdots + a_n Y^{n-1} ~|~a_i \in \mathbb{R} \} \]  
of homogeneous polynomials of degree $n-1$.
If we denote a canonical basis of $Sym^{n-1}(\mathbb{R}^2)$ by $e_1^{n-1}, e_1^{n-2} \cdot e_2, \cdots , e_2^{n-1}$, where $e_1, e_2$ are canonical basis of $\mathbb{R}^2$, the identification is defined by mapping the vector $e_1^i \cdot e_2^{n-1-i}$ to $\binom{n-1}{i}X^iY^{n-1-i}$.
Then the one dimensional subspace $\nu ([a:b])$ is equal to $\mathbb{R} <(aX+bY)^{n-1}>$ in the vector space ${\rm Poly}_n(X,Y)$. 
In addition to the flag curve associated to $\nu$, which is again denoted by $\nu$, is defined by $ \{ P(X,Y) \in {\rm Poly}_n(X,Y) ~|~ \exists Q(X,Y) ~s.t.~ P(X,Y) = (aX + bY)^{n-d} Q(X,Y) \}$, which is a $d$-dimensional subspace. 
We call this flag curve $\nu \circ f_{\rho}$ the {\it Veronese flag curve}.
This Veronese flag curve is the flag curve of Fuchsian representations.

\subsection{Anosov property}
We recall the Anosov property of representations which is strongly related to the hyperconvexity of representations.
See \cite{GGKW17} for the detail.
Let $G$ be a semisimple Lie group and $K$ be a maximal compact Lie group.
The Lie algebra of $G$, denoted by $\mathfrak{g}$, is decomposed into $\mathfrak{k} \oplus \mathfrak{p}$ by the Killing form, where $\mathfrak{k}$ is the Lie algebra of $K$.
We take a maximal abelian subalgebra $\mathfrak{a} \subset \mathfrak{p}$.
Let $\mathfrak{g} = \mathfrak{g}_0 \oplus \bigoplus_{\alpha \in \Sigma}\mathfrak{g}_{\alpha}$ be a root decomposition where $\Sigma$ is the system of restricted roots of $\mathfrak{g}$.
We denote the set of positive roots by $\Sigma^+ = \{ \alpha \in \Sigma ~|~ \alpha > 0 \}$ and the set of negative roots by $\Sigma^- = \{ \alpha \in \Sigma ~|~ \alpha < 0 \}$.
The set $\Delta \subset \Sigma^+$ is the set of simple roots.
Let $\mathfrak{n}^{\pm} = \bigoplus_{\alpha \in \Sigma^{\pm}} \mathfrak{g}_{\alpha}$ and $N^{\pm} = \exp(\mathfrak{n}^{\pm})$.
For a subset $\theta \subset \Delta$, we set $\mathfrak{a}_{\theta} = \bigcap_{\alpha \in \theta}{\rm Ker}\alpha$, and $M_{\theta} = Z_K(\mathfrak{a}_{\theta})$, the centralizer of $\mathfrak{a}_{\theta}$ in $K$.
The subgroup $P_{\theta} = M_{\theta}\exp(\mathfrak{a})N^+$ is called a parabolic subgroup of $G$.
Two parabolic subgroups are said to be opposite if their intersection is reductive.
It is known that any pair of opposite parabolic subgroups is conjugate to a pair $(P_{\theta}, P^-_{\theta})$ for a subset $\theta \subset \Delta$ where $P^-_{\theta} = M_{\theta}\exp(\mathfrak{a})N^-$.
Let $\mathfrak{a}^+ = \{ a \in \mathfrak{a} ~|~ \alpha(a) > 0 ~\forall \alpha \in \Sigma^+ \}$ which is called a Weyl chamber.
There is a decomposition of $G$ into $K \exp(\bar{\mathfrak{a}}^+) K$, called the Cartan decomposition.
In particular, any element $g \in G$, there exists $k, k' \in K$ and a unique $\mu(g) \in \bar{\mathfrak{a}}^+$ such that $g = k \exp(\mu(g))k'$.
The correspondence $\mu : G \rightarrow \bar{\mathfrak{a}}^+$ is called the Cartan projection. 

For a parabolic subgroup $P_{\theta}$, the homogeneous space $G/P_{\theta}$ is called a flag manifold.
Flag manifolds $G/P_{\theta}$ are identified with the set of conjugates of $P_{\theta}$ in $G$ which are also parabolic subgroups.
Consider two maps $\xi^+ : \partial \Gamma \rightarrow G/P_{\theta}$ and $\xi^- : \partial \Gamma \rightarrow G/P^-_{\theta}$ from the ideal boundary of a hyperbolic group $\Gamma$ into flag manifolds.
The maps $\xi, \xi^-$ are said to be transverse if $\xi^+(x)$ and $\xi^-(y)$, which are identified with parabolic subgroups, are opposite for any distinct points $x,y \in \partial \Gamma$.
Moreover they are said to be dynamics-preserving for a representation $\rho : \Gamma \rightarrow G$ if for any $\gamma \in \Gamma$ of infinite order $\xi(\gamma^+)$ and $\xi^-(\gamma^+)$ are the attracting fixed point of $\rho(\gamma)$ where $\gamma^+ \in \partial \Gamma$ is the attracting fixed point of $\gamma$.

\begin{definition}
Let $\Gamma$ be a word hyperbolic group, G a semisimple Lie group, and $\theta \subset \Delta$ a nonempty subset of the restricted roots of $G$.
A representation $\rho : \Gamma \rightarrow G$ is said to be $P_{\theta}$-Anosov if
there exists continuous, $\rho$-equivariant and transverse maps $\xi^+ : \partial \Gamma \rightarrow G/P_{\theta}$ and $\xi^- : \partial \Gamma \rightarrow  G/P_{\theta}^-$ such that the maps $\xi^{\pm}$ are dynamics-preserving for $\rho$ and 
\[ \exists c, C > 0, \forall \alpha \in \theta, \forall \gamma \in \Gamma, \alpha(\mu(\rho(\gamma))) \geq c|\gamma| - C. \]   
\end{definition} 

In \cite{La06}, Labourie showed Hitchin representations are $B$-Anosov for a Borel subgroup $B$ of ${\rm PSL}_n\mathbb{R}$, and a faithful discrete irreducible representation.
The maps $\xi^{\pm}$ are called the {\it boundary maps} of the Anosov representation $\rho$.
Since $G/B \cong {\rm Flag}(\mathbb{R}^n)$ when $B$ is Borel, the boundary maps  are maps from $\partial \pi_1(S)$ to the flag manifold ${\rm Flag}(\mathbb{R}^n)$ and coincide with the flag curve induced by the hyperconvexity of Hitchin representation.

\begin{remark}
In the definition of Anosov representations, we follow Gu\'eritaud-Guichard-Kassel-Wienhard \cite{GGKW17}.
The original definition is given by Labourie \cite{La06} and Guichard-Wienhard \cite{GW12}.
Kapovich-Leeb-Porti \cite{KLP17} gives another definition in the viewpoint of the geometry of symmetric spaces. 
\end{remark}

Here we recall the property of  eigenvalues of Hitchin representation shown by Anosov property. 
\begin{proposition}[Labourie \cite{La06}, Bonahon-Dreyer \cite{BD14}]
Let $\rho : \pi_1(S) \rightarrow {\rm PSL}_n\mathbb{R}$ be a Hitchin representation and $\gamma \in \pi_1(S)$ a nontrivial element of $\pi_1(S)$. Then $\rho(\gamma)$ has a lift $ \widetilde{\rho(\gamma)}  \in {\rm SL}_n(\mathbb{R})$ whose eigenvalues are distinct and positive.

\end{proposition}
In the setting of this proposition, we denote the eigenvalues of a lift $\widetilde{\rho(\gamma)}$ by $\lambda^{\rho}_1(\gamma) > \lambda^{\rho}_2(\gamma) >  \cdots > \lambda^{\rho}_n(\gamma) >0$. 
We define the {\it $k$-th length function} of a Hitchin representation $\rho$ by
\[ l_k^{\rho}(\gamma) = \log \dfrac{\lambda_k^{\rho}(\gamma)}{\lambda_{k+1}^{\rho}(\gamma)} \]
where $k=1, \cdots, n-1$. 
This is well-defined on the Hitchin component $H_n(S)$ since the conjugation preserves eigenvalues.
The length function of Hitchin representations is an extension of a hyperbolic length function of simple closed curves of surfaces.
This is used in the closed leaf condition in the next section.

\section{The Bonahon-Dreyer parameterization}

\subsection{Projective invariants}
We define  projective invariants of tuples of flags.
A (complete) {\it flag} in $\mathbb{R}^n$ is a sequence of nested vector subspaces of $\mathbb{R}^n$
\[ F = ( \{0\} = F^0 \subset F^1 \subset F^2 \subset \cdots \subset F^n = \mathbb{R}^n )\]
where ${\rm dim}F^d = d$. 
The {\it flag manifold} of $\mathbb{R}^n$ is a set of flags in $\mathbb{R}^n$.
We denoted the flag manifold by ${\rm Flag}(\mathbb{R}^n)$.
Note that ${\rm Flag}(\mathbb{R}^n)$ is homeomorphic to a homogeneous space  ${\rm PSL}_n\mathbb{R} / B$, where $B$ is a Borel subgroup of ${\rm PSL}_n\mathbb{R}$, and ${\rm PSL}_n\mathbb{R}$ naturally acts on the flag manifold.
A {\it generic} tuple of flags is a tuple $(F_1, F_2, \cdots, F_k)$ of a finite number of flags $F_1, F_2, \cdots, F_k \in {\rm Flag}(\mathbb{R}^n)$ such that if $n_1, \cdots, n_k$ are nonnegative integers satisfying $n_1 + \cdots + n_k = n$, then $F_1^{1} \cap \cdots \cap F_k^{n_k} = \{ 0 \}$. 

Let $(E, F, G)$ be a generic triple of flags, and $p,q,r \geq 1$ integers with $p+q+r = n$. Choose a basis $e^d, f^d, g^d$ of the wedge product spaces $\bigwedge^dE^d, \bigwedge^dF^d, \bigwedge^dG^d$, which are one dimensional subspaces, for each $d=1, \cdots , n$ respectively. 
We fix an identification between $\bigwedge^n \mathbb{R}^n$ with $\mathbb{R}$.
Then we can regard $e^{d_1} \wedge f^{d_2} \wedge g^{d_3} $ as an element of $\mathbb{R}$ since $d_1 + d_2 + d_3 = n$.
In particular $e^{d_1} \wedge f^{d_2} \wedge g^{d_3} $ is not equal to $0$ since $(E,F,G)$ is generic.
\begin{definition}
The $(p,q, r)$-th triple ratio $T_{pqr}(E,F,G)$ for a generic triple of flags $(E, F, G)$ is defined by
\[ T_{pqr}(E,F,G) = \dfrac{ e^{p+1} \wedge f^{q} \wedge g^{r-1} \cdot  e^{p} \wedge f^{q-1} \wedge g^{r+1} \cdot  e^{p-1} \wedge f^{q+1} \wedge g^{r} }{  e^{p-1} \wedge f^{q} \wedge g^{r+1} \cdot  e^{p} \wedge f^{q+1} \wedge g^{r-1} \cdot  e^{p+1} \wedge f^{q-1} \wedge g^{r} }. \]
\end{definition}
The value of $T_{pqr}(E,F,G)$ is independent of the fixed identification $\bigwedge^n \mathbb{R}^n \cong \mathbb{R}$ and the choice of elements $e^d, f^d, g^d$.
If the one of exponent of $e^d, f^d, g^d$ is equal to $0$, then we ignore the corresponding terms.
For example, $e^0 \wedge f^q \wedge g^{n-q} =  f^{q} \wedge g^{n-q}$. 
The action of ${\rm PSL}_n\mathbb{R}$ leaves the triple ratio invariant by definition.

For the permutation of $(E,F,G)$, the triple ratio behaves as below.
\begin{proposition}
For a generic tuple of flags $(E,F,G)$, 
\[ T_{pqr}(E,F,G) = T_{qrp}(F,G,E) = T_{qpr}(F,E,G)^{-1}. \]

\end{proposition}

Let $(E,F,G,G')$ be a generic quadruple of flags, and $p$ an integer with $1 \leq p \leq n-1$.
We choose nonzero elements $e^{d}, f^{d}, g^{d}$ and $g'^{d}$ respectively in $\bigwedge^{d}E^{d}, \bigwedge^{d}F^{d}, \bigwedge^{d}G^{d}$ and $\bigwedge^{d}G'^{d}$.
\begin{definition}
The $p$-th double ratio $D_p(E,F,G,G')$ is defined by
\[ D_p(E,F,G,G') = - \dfrac{e^p \wedge f^{n-p-1} \wedge g^1 \cdot e^{p-1} \wedge f^{n-p} \wedge g'^1}{e^p \wedge f^{n-p-1} \wedge g'^1 \cdot e^{p-1} \wedge f^{n-p} \wedge g^1}. \]
\end{definition}
This is also valued in the real number, well-defined, and invariant for the action of ${\rm PSL}_n\mathbb{R}$.

\subsection{Construction of invariants}
We define three kinds of invariants of Hitchin representations, {\it triangle invariant}, {\it shearing invariant}, and {\it gluing invariant} for an oriented maximal geodesic lamination with a short arc system associated to a pants decomposition.
The triangle invariant is defined for ideal triangles, induced by the ideal triangulation, using the triple ratio.
The shearing invariant is defined for biinfinite leaves using the double ratio.
The gluing invariant is defined for closed leaves using the short arc system and he double ratio.
 
We take a pants decomposition $\mathcal{P}$ of the reference hyperbolic surface $S$.
Let $ \{ C_1, \cdots, C_{\frac{3|\chi(S)| - b}{2}} \}$ be the family of decomposing curves.
If $P \in \mathcal{P}$ is bounded by $C_i, C_j, C_k$, we write the pants $P$ by $P_{ijk}$.
Consider the oriented maximal geodesic lamination $\mathcal{L} = \{ C_i, B_{ij} \}$ where $B_{ij}$ is a spiraling biinfinite geodesic connecting decomposing curves $C_i,C_j$.
In the notation above, we do not care the ordering of the indices.
For example, $B_{ij} = B_{ji}$.
The signature of each spiraling is arbitrary.
In addition to we fix a short arc system $K$ of $\mathcal{L}$.
Note that $\mathcal{L}$ induces an ideal triangulation of $S$. 
We denote the set of ideal triangles of this triangulation by $\mathcal{T} = \{T_{ijk}^0, T_{ijk}^1\}$. where $T_{ijk}^0,T_{ijk}^1$ are ideal triangles contained in a pair of pants $P_{ijk}$.
Let $\rho : \pi_1(S) \rightarrow {\rm PSL}_n\mathbb{R}$ be a Hitchin representation and $\xi_{\rho} : \partial \pi_1(S) \rightarrow \mathbb{H}^2$ the associated flag curve.

Fix a lift $\tilde{T}$ of $T \in \mathcal{T}$ and choose an ideal vertex $v_0$ of $\tilde{T}$ arbitrarily. 
We call the other ideal vertices $v_1, v_2$ so that $v_0, v_1, v_2$ are in clockwise order.
Let $p,q,r$ be integers such that $p, q, r \geq 1$ and $p+q+r=n$.

\begin{definition}
The $(p,q,r)$-th triangle invariant $\tau_{pqr}((T,v_0), \rho)$ of a Hitchin representation $\rho$ and an ideal triangle $T$ and a chosen vertex $v_0$ is defined by
\[ \tau_{pqr}((T, v_0), \rho) = \log T_{pqr}(\xi_{\rho}(v_0), \xi_{\rho}(v_1), \xi_{\rho}(v_2)). \]
\end{definition}

The triangle invariant is independent of a choice of the lift $\tilde{T}$ since flag curves are $\rho$-equivariant and the triple ratio is invariant for the ${\rm PSL}_n\mathbb{R}$-action.

A biinfinite leaf $B \in \mathcal{L}_K$ is a side of two ideal triangles.
Let $T^l$ (resp. $T^r$) be the ideal triangle which is on the left (resp. right) side with respect to the orientation of $B$. 
We lift $B$ to a geodesic $\tilde{B}$ in $\tilde{S}$, and we also lift $T^l$ and $T^r$ to two ideal triangles $\tilde{T^l}$ and $\tilde{T^r}$ so that they are adjacent along $\tilde{B}$. 
We denote the repelling point and attracting point of $\tilde{B}$ by $y$ and $x$, and denote
the other vertex of $\tilde{T}^l$ and $\tilde{T}^r$ by $z^l$ and $z^r$ respectively. 
Let $p$ be an integer with $1 \leq p \leq n-1$. 
\begin{definition}
The $p$-th shearing invariant of a Hitchin representation $\rho$ along $B$ is defined by 
\[ \sigma_p(B, \rho) = \log D_p(\xi_{\rho}(x), \xi_{\rho}(y), \xi_{\rho}(z^l), \xi_{\rho}(z^r)). \]
\end{definition}
This invariant is also well-defined for a choice of lifts by the same reason with the case of triangle invariants.

Consider a closed leaf $C \in \mathcal{L}_K$ with the short arc $K_C$. 
Let $T^l$(resp.  $T^r$) $\in \mathcal{T}$ be ideal triangles which is spiraling along $C$ from the left (resp. right) of $C$ and contains the endpoints of $K_C$. 
Lift  $C$ and $K_C$ to a geodesics $\tilde{C}$ and an arc $\tilde{K_C}$ so that $\tilde{K}_C$ intersects $\tilde{C}$.
In addition to we take lifts $\tilde{T}^l$ and $\tilde{T}^r$ of $T^l$ and $T^r$ respectively such that they contain the endpoints of $\tilde{K}_C$.
We denote, by $x$ and $y$, the repelling and attracting point of the geodesic $\tilde{C}$ respectively. 
Let us define the vertex $z^l, z^r$ of ideal triangles $\tilde{T}^l, \tilde{T}^r$ as follows. 
In the sides of $\tilde{T}^l$, two sides are asymptotic to $\tilde{C}$.
One of these sides cuts the universal cover $\tilde{S}$ such that an ideal triangle $\tilde{T}^l$ and the geodesic $\tilde{C}$ is contained in the same connected component.
The ideal vertex $z^l$ is the end of such a geodesic side of $\tilde{T}^l$ other from the ends of $x,y$.
We define $v^r$ for $\tilde{T}^r$ similarly.
Let $p$ be an integer with $1 \leq p \leq n-1$.
\begin{definition}
The $p$-th gluing invariant of a Hitchin representation $\rho$ along $C$ is defined by 
\[ \theta_p(C, \rho)= \log D_p(\xi_{\rho}(x), \xi_{\rho}(y), \xi_{\rho}(z^l), \xi_{\rho}(z^r)). \]
\end{definition}
The invariants above are well-defined on Hitchin components {\it i.e.} these invariants are independent of representatives of conjugacy class of Hitchin representations.

\subsection{Closed leaf condition}

There is a nice relation between length functions, triangle invariants and shearing invariants. 
Let $C$ be a closed leaf of the lamination $\mathcal{L}_K$. 
Let us focus on the right side of $C$ with respect to the orientation of $C$. 
Let $B_{1}, \cdots, B_{k}$ be the biinfinite leaves spiraling along $C$ from the right, and $T_{1}, \cdots, T_{k}$ the ideal triangles which spiral along $C$ from the right. 
Suppose that these leaves and triangles spiral to $C$ in the direction (resp. the opposite direction) of the orientation of $C$. 
Define $\overline{\sigma}_{p}(B_i, \rho)$ by $\sigma_p(B_i, \rho)$ if $B_i$ is oriented toward $C$, and by $\sigma_{n-p}(B_i, \rho)$ otherwise. 
Then we define

\begin{align*}
R_p^{\rho}(C) &=\sum_{i=1}^{k} \overline{\sigma}_{p}(B_i, \rho) + \sum_{i=1}^k \sum_{q+r=n-p}\tau_{pqr}((T_i, v_i), \rho), \\
(\mbox{resp.  }  R_p^{\rho}(C) &= -\sum_{i=1}^{k} \overline{\sigma}_{n-p}(B_i, \rho) - \sum_{i=1}^k \sum_{q+r=p}\tau_{(n-p)qr}((T_{i}, v_{i}), \rho) ~,)
\end{align*}

where $v_{i}$ is the ideal vertex of a lift $\tilde{T}_{i}$ of $T_i$ which is an attracting (resp. repelling) point of a lift of $C$. When we focus on the left side of $C$, we can define $L_p^{\rho}(C)$ similarly as follows.

\begin{align*}
 L_p^{\rho}(C) &=- \sum_{i=1}^{k} \overline{\sigma}_{p}(B_i, \rho) - \sum_{i=1}^k \sum_{q+r=n-p}\tau_{pqr}((T_i, v_i), \rho). \\
(\mbox{resp.  }  L_p^{\rho}(C) &= \sum_{i=1}^{k} \overline{\sigma}_{n-p}^{\rho}(B_i, \rho) + \sum_{i=1}^k \sum_{q+r=p}\tau_{(n-p)qr}((T_{i}, v_{i}), \rho) ~.)
\end{align*}

\begin{proposition}[Bonahon-Dreyer \cite{BD14}, Proposition 13]
For any $\rho \in {H}_n(S)$ and any closed leaf $C$, it holds that 
\[ l_p^{\rho}(C) = R_p^{\rho}(C) =  L_p^{\rho}(C). \]

\end{proposition}

\subsection{Bonahon-Dreyer parameterization}
We apply the Bonahon-Dreyer parameterization theorem in our case.
For the geodesic lamination $\mathcal{L}_K$, we have $\frac{3|\chi(S)|}{2}$ closed leaves $C_i$, $3|\chi(S)|$ biinfinite leaves $B_{ij}$, and $2|\chi(S)|$ ideal triangles $T_{ijk}^l$. 
Set $N = \frac{3|\chi(S)|}{2}(n-1) +  3|\chi(S)|(n-1) +  2|\chi(S)|\binom{n-1}{2}$.
By proposition 4.2, we have a relation between triangle invariants:
\begin{proposition}
\[  \tau_{pqr}((T, v_0), \rho) = \tau_{qrp}((T, v_1), \rho) = \tau_{rpq}((T, v_2), \rho). \]

\end{proposition}
Thus it is enough to consider only the triangle invariant defined for one ideal vertex and we denote the triangle invariant as $\tau_{pqr}(T, \rho)$ simply.
 Bonahon-Dreyer showed that Hitchin representations are  parameterized by the all triangle invariants, shearing invariants, and gluing invariants we can consider.
\begin{theorem}[Bonahon-Dreyer \cite{BD14}, \cite{BD17}]
The map
\begin{align*}
&\Phi_{\mathcal{L}_K} : H_n(S) \rightarrow \mathbb{R}^N \\
&\Phi_{\mathcal{L}_K}(\rho) =  (\tau_{pqr}(T_{ijk}^l, \rho), \cdots, \sigma(B_{ij},\rho), \cdots, \theta(C_i, \rho), \cdots).
\end{align*}
is a homeomorphism onto the image.
Moreover the image of this map is the interior $\mathcal{P}_{\mathcal{L}_K}$ of a convex polytope.

\end{theorem}
The parameter space is coincides with the interior of the convex polytope which is defined by  the closed leaf condition.
We denote the coordinate of the target space $\mathbb{R}^N$ by $(\tau_{pqr}(T_{ijk}^l), \cdots, \sigma(B_{ij}), \cdots, \theta(C_i), \cdots)$.

\section{Invariants of Fuchsian representations}

Let $\rho = \iota_n \circ \rho$ be a Fuchsian representation defined by a hyperbolic holonomy $\rho : \pi_1(S) \rightarrow {\rm PSL}_2\mathbb{R}$.
We denote, by $\partial \pi_1(S)^{(3)}$ (resp. $\partial \mathbb{H}^{(3)}$), the set of triples of pairwise distinct points of $\partial \pi_1(S)$ (resp. $\partial \mathbb{H}^2$.
\begin{proposition}
For any triples $(x,y,z) \in \partial \pi_1(S) ^{(3)}$ in clockwise order, the $(p,q,r)$-triple ratio $T_{pqr}(\xi_{\rho_n}(x),\xi_{\rho_n}(y),\xi_{\rho_n}(z)) = 1$.
\end{proposition}
\begin{proof}
Since ${\rm PSL}_2\mathbb{R}$ transitively acts on the set of triples $\partial \mathbb{H}^{(3)}$, we can take a transformation $A \in {\rm PSL}_2\mathbb{R}$ such that $A(f_{\rho}(x))=\infty, A(f_{\rho}(y)) = 1, A(f_{\rho}(z)) = 0$.
Using this normalization, we have
\begin{align*}
T_{pqr}(\xi_{\rho_n}(x),\xi_{\rho_n}(y),\xi_{\rho_n}(z)) 
&= T_{pqr}(\nu(f_{\rho}(x)), \nu(f_{\rho}(y)), \nu(f_{\rho}(z))) \\
&= T_{pqr}(\nu(A^{-1}(\infty)), \nu(A^{-1}(1)), \nu(A^{-1}(0)) \\
&= T_{pqr}(\iota_n(A)^{-1}\nu(\infty), \iota_n(A)^{-1}\nu(1), \iota_n(A)^{-1}\nu(0)) \\
&= T_{pqr}(\nu(\infty), \nu(1), \nu(0)).
\end{align*}
Thus it is enough to consider the value $T_{pqr}(\nu(\infty), \nu(1), \nu(0))$.

Recall that the flag $\nu([a:b]) = \{ V_d \}_d$ for $[a:b] \in \mathbb{RP}^1$ consists of the nested vector space $V_d$ of dimension $d = 0,1, \cdots, n$ defined by 
\[ V_d = \{ P(X,Y) \in {\rm Poly}_n(X,Y) ~|~ \exists Q(X,Y) ~s.t.~ P(X,Y) = (aX + bY)^{n-d} Q(X,Y) \}. \] 
For example, the $r$-dimensional vector space $\nu(0)^r$ is  
\begin{align*}
\nu(0)^d 
&= \{ P(X,Y) ~|~ \exists Q(X,Y) ~s.t.~ P(X,Y) = Y^{n-d} Q(X,Y) \} \\
&= \{ (k_1 X^{d-1} + k_2 X^{d-2} Y + \cdots + k_{d} Y^{d-1})Y^{n-d} ~|~ k_1, \cdots k_d \in \mathbb{R}   \}  \\
&= {\rm Span}<X^{d-1}Y^{n-d}, X^{d-2}Y^{n-d+1}, \cdots, Y^{n-1} > .
\end{align*}
Similarly,
\begin{align*}
\nu(\infty)^d 
&= {\rm Span}<X^{n-1}, X^{n-2}Y, \cdots, X^{n-d}Y^{d-1}>, \\
\nu(1)^d
&= {\rm Span}<(X+Y)^{n-d}X^{d-1}, (X+Y)^{n-d}X^{d-2}Y, \cdots, (X+Y)^{n-d}Y^{d-1} >.
\end{align*}
To compute the triple ratio, first we choose a basis of $\bigwedge^d \nu(0)^d, \bigwedge^d \nu(1)^d, \bigwedge^d \nu(\infty)^d$ as follows:
\begin{align*}
t_0^d &=  X^{d-1}Y^{n-d} \wedge X^{d-2}Y^{n-d+1} \wedge \cdots \wedge Y^{n-1} \in \bigwedge^d \nu(0)^d, \\
t_{\infty}^d &= X^{n-1} \wedge X^{n-2}Y \wedge \cdots \wedge X^{n-d}Y^{d-1} \in \bigwedge^d \nu(\infty)^d, \\
t_1^d &= (X+Y)^{n-d}X^{d-1} \wedge (X+Y)^{n-d}X^{d-2}Y \wedge \cdots \wedge (X+Y)^{n-d}Y^{d-1} \in \bigwedge^d \nu(1)^d.
\end{align*}
Then $T_{pqr}(\nu(\infty), \nu(1), \nu(0))$ is precisely equal to
\[ \dfrac
{
t_{\infty}^{p+1} \wedge t_1^{q} \wedge t_0^{r-1} \cdot t_{\infty}^{p} \wedge t_1^{q-1} \wedge t_0^{r+1} \cdot t_{\infty}^{p-1} \wedge t_1^{q+1} \wedge t_0^{r}
}
{
t_{\infty}^{p-1} \wedge t_1^{q} \wedge t_0^{r+1} \cdot t_{\infty}^{p} \wedge t_1^{q+1} \wedge t_0^{r-1} \cdot t_{\infty}^{p+1} \wedge t_1^{q-1} \wedge t_0^{r}
},
\]
so we should verify values of wedge products $t_{\infty}^{p} \wedge t_1^{q} \wedge t_0^{r}$ for integers $p,q,r$ with $0 \leq p,q,r \leq n$ and $p+q+r = n$.
(We abuse notion $p,q,r$ which appeared in the statement of proposition 5.1.)
The following formula is shown by easy linear algebra.
\begin{lemma}
Let $V$ be an $n$-dimensional vector space with a basis $\{b_1, \cdots, b_n \}$ and $\{v_1, \cdots, v_n\}$ be arbitrary vectors in $V$.
If $v_i = \sum_{i=1}^n v_{ij}b_j$, then 
\[v_1 \wedge \cdots \wedge v_n = {\rm Det}((v_{ij}))b_1 \wedge \cdots \wedge b_n. \]

\end{lemma}
We fix a basis of ${\rm Poly}_n(X,Y)$ by $b_1 = X^{n-1}, b_2 = X^{n-2}Y, \cdots, b_n = Y^{n-1}$, and we may choose an identification $\bigwedge^n {\rm Poly}_n(X,Y) \rightarrow \mathbb{R}$ such that $b_1 \wedge b_2 \wedge \cdots \wedge b_n$ is identified with 1. 
Then, using this basis, 
\begin{align*}
t_{\infty}^{p} \wedge t_1^{q} \wedge t_0^{r} 
&= 
X^{n-1} \wedge  X^{n-2}Y \wedge \cdots \wedge X^{n-p}Y^{p-1} \wedge \\
&\qquad
(X+Y)^{n-q}X^{q-1} \wedge  (X+Y)^{n-q}X^{q-2}Y \wedge \cdots \wedge (X+Y)^{n-q}Y^{q-1}  \wedge \\
&\qquad \quad
X^{r-1}Y^{n-r} \wedge X^{r-2}Y^{n-r+1} \wedge \cdots \wedge Y^{n-1} 
\\
&= 
b_1 \wedge b_2 \wedge \cdots b_p \wedge \\
& \qquad 
\sum_{i=1}^{n-q}\binom{n-q}{i} b_{i+1} \wedge \sum_{i=1}^{n-q}\binom{n-q}{i} b_{i+2} \wedge \cdots \wedge \sum_{i=1}^{n-q}\binom{n-q}{i} b_{i+q} \wedge \\
&\qquad \quad 
b_{n-r+1} \wedge b_{n-r+2} \wedge \cdots  \wedge b_n.
\end{align*}
By Lemma 5.2 and an easy computation of determinant of matrix, we get
\[ t_{\infty}^{p} \wedge t_1^{q} \wedge t_0^{r}  =
\begin{vmatrix}
\binom{p+r}{p} & \cdots & \binom{p+r}{p-q+1} \\
\vdots & \vdots & \vdots \\
\binom{p+r}{p+q-1} & \cdots & \binom{p+r}{p} \\
\end{vmatrix} \]
if $q \neq 0$ and $t_{\infty}^{p} \wedge t_1^{0} \wedge t_0^{r} = 1$.
We suppose $q \neq 0$.
Note that we now consider an extended binomial coefficient defined by
\[ \binom{n}{p} = 
\begin{cases}
\dfrac{n!}{p!(n-p)!}  & (0 \leq p \leq n) \\
0 & (otherwise).
\end{cases} \]
Hence many zero entries may appear in this determinant.
\begin{lemma}
The determinant
\[
\begin{vmatrix}
\binom{p+r}{p} & \cdots & \binom{p+r}{p-q+1} \\
\vdots & \vdots & \vdots \\
\binom{p+r}{p+q-1} & \cdots & \binom{p+r}{p} \\
\end{vmatrix}
\]
is equal to
\[
(-1)^\frac{(q-1)q}{2}\dfrac{(n-q)!(n-q+1)! \cdots (n-1)! 1!2! \cdots (q-1)!}{(n-r-q)!(n-r-q+1)! \cdots (n-r-1)! r! (r+1)! \cdots (r+q-1)!}.
\]
\end{lemma}
\begin{proof}[Proof of Lemma 5.3.]
The following formulae still hold for the definition of the extended binomial coefficient.
\begin{align}
\binom{n}{p} &= \binom{n}{n-p}, \\
\binom{n}{p} + \binom{n}{p+1} &= \binom{n+1}{p+1}.
\end{align}
By the elemental transformations of matrices, adding the second row to the first row, the third row to the second row, ..., and the $q$-th row to the $(q-1)$-th row and using the formula (2), we get
\[
\begin{vmatrix}
\binom{p+r}{p} & \cdots & \binom{p+r}{p-q+1} \\
\vdots & \vdots & \vdots \\
\binom{p+r}{p+q-1} & \cdots & \binom{p+r}{p} \\
\end{vmatrix}
 =
\begin{vmatrix}
\binom{p+r+1}{p+1} & \cdots & \binom{p+r+1}{p-q+2} \\
\binom{p+r+1}{p+2} & \cdots & \binom{p+r+1}{p-q+3} \\
\binom{p+r+1}{p+3} & \cdots & \binom{p+r+1}{p-q+4} \\
\vdots & \vdots & \vdots \\
\binom{p+r+1}{p+q-2} & \cdots & \binom{p+r+1}{p-1} \\
\binom{p+r+1}{p+q-1} & \cdots & \binom{p+r+1}{p} \\
\binom{p+r}{p+q-1} & \cdots & \binom{p+r}{p} \\
\end{vmatrix} .\] 
Next, by adding the second row to the first row, the third row to the second row, ..., and the $(q-1)$-th row to the $(q-2)$-th row and using the formula (2),
\[
\begin{vmatrix}
\binom{p+r+1}{p+1} & \cdots & \binom{p+r+1}{p-q+2} \\
\binom{p+r+1}{p+2} & \cdots & \binom{p+r+1}{p-q+3} \\
\binom{p+r+1}{p+3} & \cdots & \binom{p+r+1}{p-q+4} \\
\vdots & \vdots & \vdots \\
\binom{p+r+1}{p+q-2} & \cdots & \binom{p+r+1}{p-1} \\
\binom{p+r+1}{p+q-1} & \cdots & \binom{p+r+1}{p} \\
\binom{p+r}{p+q-1} & \cdots & \binom{p+r}{p} \\
\end{vmatrix}
 =
\begin{vmatrix}
\binom{p+r+2}{p+2} & \cdots & \binom{p+r+2}{p-q+3} \\
\binom{p+r+2}{p+3} & \cdots & \binom{p+r+2}{p-q+4} \\
\binom{p+r+2}{p+4} & \cdots & \binom{p+r+2}{p-q+5} \\
\vdots & \vdots & \vdots \\
\binom{p+r+2}{p+q-1} & \cdots & \binom{p+r+2}{p-1} \\
\binom{p+r+1}{p+q-1} & \cdots & \binom{p+r+1}{p} \\
\binom{p+r}{p+q-1} & \cdots & \binom{p+r}{p} \\
\end{vmatrix} .\] 
Iterating this deformation, we get
\[ 
\begin{vmatrix}
\binom{p+r}{p} & \cdots & \binom{p+r}{p-q+1} \\
\vdots & \vdots & \vdots \\
\binom{p+r}{p+q-1} & \cdots & \binom{p+r}{p} \\
\end{vmatrix}
 =
\begin{vmatrix}
\binom{p+r+q-1}{p+q-1} & \cdots & \binom{p+r+q-1}{p} \\
\binom{p+r+q-2}{p+q-1} & \cdots & \binom{p+r+q-2}{p} \\
\binom{p+r+q-3}{p+q-1} & \cdots & \binom{p+r+q-3}{p} \\
\vdots & \vdots & \vdots \\
\binom{p+r+2}{p+q-1} & \cdots & \binom{p+r+2}{p} \\
\binom{p+r+1}{p+q-1} & \cdots & \binom{p+r+1}{p} \\
\binom{p+r}{p+q-1} & \cdots & \binom{p+r}{p} \\
\end{vmatrix} 
=
\begin{vmatrix}
\binom{n-1}{p+q-1} & \cdots & \binom{n-1}{p} \\
\binom{n-2}{p+q-1} & \cdots & \binom{n-2}{p} \\
\binom{n-3}{p+q-1} & \cdots & \binom{n-3}{p} \\
\vdots & \vdots & \vdots \\
\binom{n-q+2}{p+q-1} & \cdots & \binom{n-q+2}{p} \\
\binom{n-q+1}{p+q-1} & \cdots & \binom{n-q+1}{p} \\
\binom{n-q}{p+q-1} & \cdots & \binom{n-q}{p} \\
\end{vmatrix} .
\] 
Note that $p+q+r =n$ for the last equality. 
We consider a similar deformation for columns.
By adding the second column to the first column, the third column to the second column, ..., and the $q$-th column to the $(q-1)$-th column, and using the formula (2), the determinant above is deformed to
\[ 
\begin{vmatrix}
\binom{n}{p+q-1}&\binom{n}{p+q-2}&\binom{n}{p+q-3} & \cdots &\binom{n}{p+2}&\binom{n}{p+1} & \binom{n-1}{p} \\
\vdots & \vdots & \vdots &~& \vdots & \vdots & \vdots \\
\binom{n-q+1}{p+q-1}&\binom{n-q+1}{p+q-2}&\binom{n-q+1}{p+q-3} & \cdots &\binom{n-q+1}{p+2}&\binom{n-q+1}{p+1} & \binom{n-q}{p} \\
\end{vmatrix} .
\] 
By adding the second column to the first column, the third column to the second column, ..., and the $(q-1)$-th column to the $(q-2)$-th column, and using the formula (2),  the determinant is again deformed to the following form
\[ 
\begin{vmatrix}
\binom{n+1}{p+q-1}&\binom{n+1}{p+q-2}&\binom{n+1}{p+q-3} & \cdots &\binom{n+1}{p+2}&\binom{n}{p+1} & \binom{n-1}{p} \\
\vdots & \vdots & \vdots &~& \vdots & \vdots & \vdots \\
\binom{n-q+2}{p+q-1}&\binom{n-q+2}{p+q-2}&\binom{n-q+2}{p+q-3} & \cdots &\binom{n-q+2}{p+2}&\binom{n-q+1}{p+1} & \binom{n-q}{p} \\
\end{vmatrix} .
\] 
By iterating this deformation, the original determinant can be deformed to the following one:
\[ 
\begin{vmatrix}
\binom{n+q-2}{p+q-1}&\binom{n+q-3}{p+q-2}&\binom{n+q-4}{p+q-3} & \cdots &\binom{n+1}{p+2}&\binom{n}{p+1} & \binom{n-1}{p} \\
\vdots & \vdots & \vdots &~& \vdots & \vdots & \vdots \\
\binom{n-1}{p+q-1}&\binom{n-2}{p+q-2}&\binom{n-3}{p+q-3} & \cdots &\binom{n-q+2}{p+2}&\binom{n-q+1}{p+1} & \binom{n-q}{p} \\
\end{vmatrix} .
\] 
Using $p+q+r=n$, and replacing columns and rows, the determinant above can be deformed as follows.

\begin{align*}
& \hspace{5pt}
\begin{vmatrix}
\binom{n+q-2}{p+q-1}&\binom{n+q-3}{p+q-2}& \cdots &\binom{n}{p+1} & \binom{n-1}{p} \\
\binom{n+q-3}{p+q-1}&\binom{n+q-4}{p+q-2}& \cdots &\binom{n-1}{p+1} & \binom{n-2}{p} \\
\vdots & \vdots & \vdots &\vdots & \vdots \\
\binom{n}{p+q-1}&\binom{n-1}{p+q-2}& \cdots &\binom{n-q+2}{p+1} & \binom{n-q+1}{p} \\
\binom{n-1}{p+q-1}&\binom{n-2}{p+q-2}& \cdots &\binom{n-q+1}{p+1} & \binom{n-q}{p} \\
\end{vmatrix} \\
&=(-1)^{\frac{q(q-1)}{2}}
\begin{vmatrix}
\binom{n-1}{p+q-1}&\binom{n-2}{p+q-2}& \cdots &\binom{n-q+1}{p+1} & \binom{n-q}{p} \\
\binom{n}{p+q-1}&\binom{n-1}{p+q-2}& \cdots &\binom{n-q+2}{p+1} & \binom{n-q+1}{p} \\
\vdots & \vdots & \vdots &\vdots & \vdots \\
\binom{n+q-3}{p+q-1}&\binom{n+q-4}{p+q-2}& \cdots &\binom{n-1}{p+1} & \binom{n-2}{p} \\
\binom{n+q-2}{p+q-1}&\binom{n+q-3}{p+q-2}& \cdots &\binom{n}{p+1} & \binom{n-1}{p} \\
\end{vmatrix} \\
&=(-1)^{\frac{q(q-1)}{2}}\cdot(-1)^{\frac{q(q-1)}{2}}
\begin{vmatrix}
\binom{n-q}{p} & \binom{n-q+1}{p+1} & \cdots & \binom{n-2}{p+q-2} & \binom{n-1}{p+q-1} \\
\binom{n-q+1}{p} & \binom{n-q+2}{p+1} & \cdots & \binom{n-1}{p+q-2} & \binom{n}{p+q-1} \\
\vdots & \vdots & \vdots &\vdots & \vdots \\
\binom{n-2}{p} & \binom{n-1}{p+1} & \cdots & \binom{n+q-4}{p+q-2} & \binom{n+q-3}{p+q-1} \\
\binom{n-1}{p} & \binom{n}{p+1} & \cdots & \binom{n+q-3}{p+q-2} & \binom{n+q-2}{p+q-1} \\ 
\end{vmatrix} \\
&=
\begin{vmatrix}
\binom{n-q}{n-r-q} & \binom{n-q+1}{n-r-q+1} & \cdots & \binom{n-2}{n-r-2} & \binom{n-1}{n-r-1} \\
\binom{n-q+1}{n-r-q} & \binom{n-q+2}{n-r-q+1} & \cdots & \binom{n-1}{n-r-2} & \binom{n}{n-r-1} \\
\vdots & \vdots & \vdots &\vdots & \vdots \\
\binom{n-2}{n-r-q} & \binom{n-1}{n-r-q+1} & \cdots & \binom{n+q-4}{n-r-2} & \binom{n+q-3}{n-r-1} \\
\binom{n-1}{n-r-q} & \binom{n}{n-r-q+1} & \cdots & \binom{n+q-3}{n-r-2} & \binom{n+q-2}{n-r-1} \\ 
\end{vmatrix}
\hspace{2pt}
\cdots (\dag).
\end{align*}
Lemma 5.3 is obtained by applying the following lemma.
The determinant $\Diamond(n,k,l)$ below corresponds to a rhombus in Pascal's triangle.
The entries of $\Diamond(n,k,l)$ are usual binomial coefficients, so positive integers.  
We can apply the formula in Lemma 5.4 to compute $(\dag)$ by replacing $n,k,l$ to $n-q, n-r-q, q-1$, and we get Lemma 5.3.
\end{proof}
\begin{lemma}
Let $n,l \in \mathbb{N}$ and $0 \leq k \leq n$.
The determinant
\[ 
\Diamond(n,k,l)=
\begin{vmatrix}
\binom{n}{k} & \binom{n+1}{k+1} &\cdots& \binom{n+l}{k+l} \\
\binom{n+1}{k} & \binom{n+2}{k+1} &\cdots & \binom{n+l+1}{k+l} \\
\vdots & \vdots & \vdots & \vdots \\
\binom{n+l}{k} & \binom{n+l+1}{k+1} & \cdots & \binom{n+2l}{k+l} \\
\end{vmatrix}
\]
is equal to
\[ \frac{n!(n+1)! \cdots (n+l)!}{k! (k+1)! \cdots (k+l)! (n-k)! \cdots (n-k+l)!} \cdot (-1)^{\frac{l(l+1)}{2}} 1! \cdots l!. \]
\end{lemma}

\begin{proof}[Proof of Lemma 5.4.]
First, we deform $\Diamond(n,k,l)$ as follows.
\begin{align*}
&
\Diamond(n,k,l) =
\begin{vmatrix}
\frac{n!}{k!(n-k)!} & \frac{(n+1)!}{(k+1)!(n-k)!} &\cdots & \frac{(n+l)!}{(k+l)!(n-k)!} \\
\frac{(n+1)!}{k!(n-k+1)!} & \frac{(n+2)!}{((k+1)!(n-k+1)!} &\cdots &  \frac{(n+l+1)!}{(k+l)!(n-k+1)!} \\
\vdots & \vdots & \vdots & \vdots \\
\frac{(n+l)!}{k!(n-k+l)!} & \frac{(n+l+1)!}{(k+1)!(n-k+l)!} & \cdots & \frac{(n+2l)!}{(k+l)!(n-k+l)!} \\      
\end{vmatrix} \\
&=C
\begin{vmatrix}
1 &1 & \cdots &1 \\
(n+1) & (n+2) & \cdots & (n+l+1) \\
\vdots & \vdots & \vdots & \vdots \\
(n+1) \cdots (n+l) & (n+2) \cdots (n+l+1) & \cdots & (n+l+1) \cdots (n+2l) \\
\end{vmatrix},
\end{align*}
where 
\[ C =  \dfrac{n!(n+1)! \cdots (n+p)!}{k! (k+1)! \cdots (k+l)! (n-k)! \cdots (n-k+l)!}. \]
We add the $(-l+1)$ times of the $l$-th row to the $(l+1)$-th row, the $(-l+2)$ times of the $(l-1)$-th row to the $l$-th row, ..., and ($-1$) times of the second row to the third row:
\begin{align*}
&
\begin{vmatrix}
1 &1 & \cdots &1 \\
(n+1) & (n+2) & \cdots & (n+l+1) \\
\vdots & \vdots & \vdots & \vdots \\
(n+1) \cdots (n+l) & (n+2) \cdots (n+l+1) & \cdots & (n+l+1) \cdots (n+2l) \\
\end{vmatrix} \\
& \qquad =
\begin{vmatrix}
1 &1 & \cdots &1 \\
(n+1) & (n+2) & \cdots & (n+l+1) \\
\vdots & \vdots & \vdots & \vdots \\
(n+1)^2 \cdots (n+l) & (n+2)^2 \cdots (n+l+1) & \cdots & (n+l+1)^2 \cdots (n+2l) \\
\end{vmatrix}.
\end{align*}
The iteration of such a deformation gives us the following determinant:
\[
\begin{vmatrix}
1 &1 & \cdots &1 \\
(n+1) & (n+2) & \cdots & (n+l+1) \\
\vdots & \vdots & \vdots & \vdots \\
(n+1)^l & (n+2)^l& \cdots & (n+l+1)^l \\
\end{vmatrix}.
\]
We can use the formula of Vandermonde's determinant and expand this as follows.
\begin{align*}
\begin{vmatrix}
1 &1 & \cdots &1 \\
(n+1) & (n+2) & \cdots & (n+l+1) \\
\vdots & \vdots & \vdots & \vdots \\
(n+1)^l & (n+2)^l& \cdots & (n+l+1)^l \\
\end{vmatrix} 
&= (-1)^l l! \cdot (-1)^{l-1} (l-1)! \cdots (-1)\\
&= (-1)^{l +(l-1) + \cdots +1} l! (l-1)! \cdots 1 \\
&= (-1)^{\frac{l(l+1)}{2}} 1! \cdots l!.
\end{align*}
Thus 
\[
\Diamond(n,k,l)
=
\frac{n!(n+1)! \cdots (n+l)!}{k! (k+1)! \cdots (k+l)! (n-k)! \cdots (n-k+l)!} \cdot (-1)^{\frac{l(l+1)}{2}} 1! \cdots l!. \]
\end{proof}
Finally, applying Lemma 5.3 to $T_{pqr}(\nu(\infty), \nu(1), \infty(0))$, we can check the value is equal to $1$.
Therefore the triple ratio of ordered triple is always equal to $1$ when we consider the Veronese flag curve. 
We finish the proof of proposition 5.1.
\end{proof}

\begin{proposition}
Let $(a,b,c,d) \in \partial\pi_1(S)^{(4)}$ be a quadruple in counterclockwise order.
Then $p$-th double ratio  $D_p(\xi_{\rho_n}(a), \xi_{\rho_n}(c), \xi_{\rho_n}(b), \xi_{\rho_n}(d))$ is equal to $-z^{-1}$, where $z=z(f_{\rho}(c),f_{\rho}(d),f_{\rho}(a),f_{\rho}(b))$ is the cross ratio of the quadruple $(f_{\rho}(c),f_{\rho}(d),f_{\rho}(a),f_{\rho}(b))$.
\end{proposition}

\begin{proof}
The proof is similar to one of proposition 5.1.
Let $A \in {\rm PSL}_2\mathbb{R}$ be a transformation which sends $f_{\rho}(c)$ to $0$, $f_{\rho}(d)$ to $1$, $f_{\rho}(a)$ to $\infty$.
The transformation $A$ maps $f_{\rho}(b)$ to the cross ratio $z= z(f_{\rho}(c),f_{\rho}(d),f_{\rho}(a),f_{\rho}(b))$.
Then, by the same computation with the case of triple ratio,
\[
D_p(\xi_{\rho_n}(a), \xi_{\rho_n}(c), \xi_{\rho_n}(b), \xi_{\rho_n}(d))
=
D_p(\nu(\infty), \nu(0), \nu(z), \nu(1)).
\]
The flags $\nu(\infty), \nu(0), \nu(1), \nu(z)$ is defined by the following vector spaces:
\begin{align*}
\nu(\infty)^d &= {\rm Span}< b_1, b_2, \cdots, b_d> \\
\nu(0)^d &= {\rm Span}<b_{n-d+1}, b_{n-d+2}, \cdots, b_n > \\
\nu(1)^1 &= \mathbb{R} \sum_{i=0}^{n-1}\binom{n-1}{i}b_{i+1} \\
\nu(z)^1 &=  \mathbb{R} \sum_{i=0}^{n-1}\binom{n-1}{i}z^{n-1-i}b_{i+1} 
\end{align*}
where $b_1, \cdots, b_n$ are the basis of ${\rm Poly}_n(X,Y)$ we used.
We choose a basis of the wedge products of the vector spaces $\bigwedge^d \nu(\infty)^d, \bigwedge^d \nu(0)^d, \nu(1)^1, \nu(z)^1$ as follows:
\begin{align*}
s_{\infty}^d &= b_1 \wedge b_2 \wedge \cdots \wedge b_d \in \bigwedge^d \nu(\infty)^d \\
s_0^d &= b_{n-d+1} \wedge b_{n-d+2} \wedge \cdots \wedge b_n \in  \bigwedge^d \nu(0)^d \\
s_1^1 &= \sum_{i=0}^{n-1}\binom{n-1}{i}b_{i+1} \in \nu(1)^1 \\
s_z^1 &= \sum_{i=0}^{n-1}\binom{n-1}{i}z^{n-1-i}b_{i+1}  \in \nu(z)^1
\end{align*}
Recall that the double ratio $D_p(\nu(\infty), \nu(0), \nu(z), \nu(1))$ is defined by
\[ D_p(\nu(\infty), \nu(0), \nu(z), \nu(1)) = - \dfrac{ s_{\infty}^p \wedge s_0^{n-p-1} \wedge s_z^1 \cdot s_{\infty}^{p-1} \wedge s_0^{n-p} \wedge s_1^1}{s_{\infty}^p \wedge s_0^{n-p-1} \wedge s_1^1 \cdot s_0^{p-1} \wedge s_0^{n-p} \wedge s_z^1} \]

Thus it is enough to compute each factors of this fraction.
The computation is very simple: 

\begin{align*}
s_{\infty}^p \wedge s_0^{n-p-1} \wedge s_z^1
&=
\begin{vmatrix}
{\rm Id}_p &  \mbox{\LARGE 0} & \binom{n-1}{0}z^{n-1} \\
 & &   \binom{n-1}{1}z^{n-2} \\
 & & \vdots \\ 
\mbox{\LARGE 0} &  {\rm Id}_{n-p-1} & \binom{n-1}{n-1}z^{0}
\end{vmatrix}\\
&=
(-1)^{n-p-1}\binom{n-1}{p}z^{n-p-1}, \\
\end{align*}

\begin{align*}
s_{\infty}^p \wedge s_0^{n-p-1} \wedge s_1^1
&=
\begin{vmatrix}
{\rm Id}_p &  \mbox{\LARGE 0} & \binom{n-1}{0} \\
 & &   \binom{n-1}{1} \\
 & & \vdots \\ 
\mbox{\LARGE 0} &  {\rm Id}_{n-p-1} & \binom{n-1}{n-1}
\end{vmatrix}\\
&=
(-1)^{n-p-1}\binom{n-1}{p} .\\
\end{align*}
\end{proof}

\begin{theorem}
If $\rho_n = \iota_n \circ \rho : \pi_1(S) \rightarrow {\rm PSL}_n\mathbb{R}$ is a ${\rm PSL}_n\mathbb{R}$-Fuchsian representation, then it follows that
\begin{itemize}
\item[(i)] all triangle invariants $\tau_{pqr}(T_{ijk}, \rho_n)$ are equal to $0$,
\item[(ii)] all shearing invariants $\sigma_p(B_{ij}, \rho_n)$ and all gluing invariants  $\theta_p(C_i,\rho_n)$ are independent of the index $p$. 
\end{itemize}

\end{theorem}
\begin{proof}
(i)
Recall the definition of triangle invariants.
Fix a lift $\tilde{T}_{ijk}$ of an ideal triangle $T_{ijk}$.
Let $x,y,z \in \partial \pi_1(S)$ be the vertices of $\tilde{T}_{ijk}$ which are in clock-wise ordering.
Then $\tau_{pqr}(T_{ijk}, \rho_n) = \log ( T_{pqr}(\xi_{\rho_n}(x), \xi_{\rho_n}(y) , \xi_{\rho_n}(z)))$.
In this case, the triple ratio is equal to $1$ by proposition 5.1, so $\tau_{pqr}(T_{ijk}, \rho_n)=0$. 

(ii)Let $\tilde{B}_{ij}$ be a lift of a biinfinite leaf $B_{ij}$ and $\tilde{T}^l$ and $\tilde{T}^r$ be lifts of the left $T^l$ and right triangles $T^r$ respectively.
Respecting the orientation of $\tilde{B}_{ij}$, we label $x,y,z^l,z^r$ on the ideal vertices of $\tilde{T}^l, \tilde{T}^r$ as in Section 4.2.
Then the quadruple $(x, z^l, y, z^r)$ is clock-wisely ordered, so by proposition 5.5, 
\begin{align*}
\sigma_p(B_{ij}, \rho_n) &= \log D_p(\xi_{\rho_n}(x),\xi_{\rho_n}(y),\xi_{\rho_n}(z^l),\xi_{\rho_n}(z^r)) \\
&= \log -z(f_{\rho}(y),f_{\rho}(z^r),f_{\rho}(x),f_{\rho}(z^l))^{-1}. 
\end{align*}
Especially, the shearing invariant is independent of the index $p$.
We can similarly show the case of gluing invariants.
The differences are only in the choice of ideal triangles and a quadruple of ideal vertices which are used in the definition of the gluing invariants.
\end{proof}
\begin{corollary}
The shearing invariants $\sigma_p(B_{ij}, \rho_n)$ of a Fuchsian representation $\rho_n = \iota_n \circ \rho$ is equal to the shearing parameter along the biinfinite leaf $B_{ij}$ defined by $\rho$.
\end{corollary} 

\begin{proof}
We have
$
\sigma_p(B_{ij}, \rho_n) = \log -z(f_{\rho}(y),f_{\rho}(z^r),f_{\rho}(x),f_{\rho}(z^l))^{-1},
$
and this is equal to the shearing parameter along $B_{ij}$ by proposition 2.4.
\end{proof}

\section{Fuchsian locus is a slice.}
Let $\mathcal{S}_{\mathcal{L}_K}$ be a slice of the convex polytope $\mathcal{P}_{\mathcal{L}_K}$, the Bonahon-Dreyer parameter space, defined by $\tau_{pqr}(T^0_{ijk}), \tau_{pqr}(T^1_{ijk}) = 0$, $\sigma_p(B_{ij}) = \sigma_q(B_{ij})$ and $\theta_p(C_i) = \theta_q(C_i)$.
By theorem 5.6, the image of the Fuchsian locus $F_n(S)$ by the Bonahon-Dreyer parameterization $\Phi_{\mathcal{L}_K}$ is contained in this slice $\mathcal{S}_{\mathcal{L}_K}$.

\begin{theorem}
The restriction map $\Phi_{\mathcal{L}_K} : F_n(S) \rightarrow \mathcal{S}_{\mathcal{L}_K}$ is surjective.
\end{theorem}
\begin{proof}
Suppose that $z_{ij}, w_i \in \mathbb{R}$ give a point of the parameter space $\mathcal{S}_{\mathcal{L}_K}$ by $\tau_{pqr}(T^0_{ijk}) , \tau_{pqr}(T^1_{ijk})= 0, \sigma_p(B_{ij}) = z_{ij}, \theta_p(C_i) = w_i$.
We construct a hyperbolic structure of $S$ with a holonomy $\eta$ such that $\Phi_{\mathcal{L}_K}(\iota_n \circ \eta)$ is equal to the point of $\mathcal{S}_{\mathcal{L}_K}$.
We show this only in the case of the maximal geodesic lamination of type (I) because the argument is completely similar.

First we focus on each pants which is given by the pants decomposition by $\{C_i\}$.
Let $P_{ijk}$ be a pants bounded three closed geodesic leaf $C_i, C_j, C_k$.
By proposition 2.5, the hyperbolic structure of a pair of pants is uniquely determined by the shearing parameters along biinfinite leaves $B_{ij}, B_{jk}, B_{ki}$.
We endows with $P_{ijk}$ the hyperbolic structure $\rho_{ijk}$ defined by $\sigma^{\rho_{ijk}}(B_{ij}) = z_{ij}, \sigma^{\rho_{ijk}}(B_{jk}) = z_{jk}, \sigma^{\rho_{ijk}}(B_{ki}) = z_{ki}$.
Since $z_{ij}, z_{jk}, z_{ki}$ satisfy the closed leaf condition, they are indeed in the image of the shearing parameterization $\sigma_{\mathcal{L}_K}$.
  
Now we glue the hyperbolic structures of $P_{ijk}$.
The hyperbolic structures of each $P_{ijk}$ gives the length of closed leaves $C_i$.
For example, the length of $C_i$ is given by $l(C_i) = |z_{ij} + z_{ik}|$, see Proposition 2.6.
We take a hyperbolic structure of $S$ such that the hyperbolic length of $C_i$ is equal to given $l(C_i)$ by the Fenchel-Nielsen coordinate of $S$ associated to the pants decomposition along $\{C_i\}$.
Let $\rho : \pi_1(S) \rightarrow {\rm PSL}_2\mathbb{R}$ be a hyperbolic holonomy of this hyperbolic structure and $f_{\rho} : \tilde{S} \rightarrow \mathbb{H}^2$ be the developing map. 

Fix a lift $\tilde{C}_i$ of the closed geodesic $C_i$.
We lift the short transverse arc $K_i = K_{C_i}$ to $\tilde{K}_i$ so that $\tilde{K}_i$ intersects to $\tilde{C}_i$.
The endpoints of $\tilde{K}_i$ are contained in two plaque $Q_i^l,Q_i^r$.
They are lifts of one of $T^0_{ijk},T^1_{ijk}, T^0_{ilm}, T^1_{ilm}$ which are ideal triangles spiraling to $C_i$.
We may assume that $Q_i^l$ is on the left and $Q_i^r$ is on the right with respect to the orientation of $C_i$. 
Let $x_i$ and $y_i$ be the starting and terminal points respectively.
Choose ideal points $z_i^l$ and $z_i^r$ of plaques $Q_i^l$ and $Q_i^r$ respectively as in Section 4.2.

We deform the hyperbolic structure $\rho$ of $S$ to a hyperbolic structure $\eta$ which realizing the following equation 
\[ \log -z(f_{\eta}(y),f_{\eta}(z_i^r),f_{\eta}(x),f_{\eta}(z_i^l))^{-1} = w_i \] 
by twist deformation.

\begin{lemma}
For any $r \in \mathbb{R}_{<0}$, there is a twist deformation $\eta$ of $\rho$ along $C_i$ such that $z(f_{\eta}(y),f_{\eta}(z_i^r),f_{\eta}(x),f_{\eta}(z_i^l)) = r$.
\end{lemma}
\begin{proof}[Proof of Lemma 6.2.]
Consider the geodesic lamination $\mathscr{C}_i$ which consists of the preimage of $C_i$ by the covering map $f_{\rho}(\tilde{S}) \rightarrow S_{\rho}$, where $S_{\rho}$ is the surface with the hyperbolic structure $\rho$. 
Let $R_i^l$ and $R_i^r$ be plaques of $\mathscr{C}_i$ containing $Q_i^l$ and $Q_i^r$ respectively.
Set $\tilde{C}_i = R_i^l \cap R_i^r$.
We observe the behavior of cross ratio under the twist deformation along $\mathscr{C}_i$.
The twists along leaves of $\mathscr{C}_i$ other from $\tilde{C}_i$ do not change the cross ratio $z(f_{\rho}(y),f_{\rho}(z_i^r),f_{\rho}(x),f_{\rho}(z_i^l))$ since the twists act on the quadruple by isometry.
Only the twist along  $\tilde{C}_i$ change the cross ratio to $z( f_{\rho}(y),f_{\rho}(z_i^r),f_{\rho}(x), {\rm Tw}_t \circ f_{\rho}(z_i^l))$, where that the map ${\rm Tw}_t$ is the extension of the twist deformation onto the ideal boundary of $\mathbb{H}^2$.
Since 
\[ \lim_{t \rightarrow \infty}{\rm Tw}_t \circ f_{\rho}(z_i^l) = -\infty, ~~ \lim_{t \rightarrow -\infty}{\rm Tw}_t \circ f_{\rho}(z_i^l) = 0, \]
there exists $t_0$ such that  $z( f_{\rho}(y),f_{\rho}(z_i^r),f_{\rho}(x), {\rm Tw}_{t_0} \circ f_{\rho}(z_i^l)) = r$ for given $r$.
\end{proof}
Using this lemma, we can deform $\rho$ to $\eta$ which satisfies for each $i$ that
\[ \log -z(f_{\eta}(y),f_{\eta}(z_i^r),f_{\eta}(x),f_{\eta}(z_i^l))^{-1} = w_i. \] 
In particular we apply Lemma 6.2 for $r = -e^{-w_i}$.
In this deformation, we should check two twist deformations along distinct curves $C_i$ and $C_j$ do not change the gluing invariant each other.
\begin{lemma}
We suppose that the hyperbolic structure $\rho$ is deformed to a hyperbolic structure $\rho_i$ by a twist deformation along $C_i$.
The twist deformation along $C_j$ does not change the cross ratio $z( f_{\rho_i}(y),f_{\rho_i}(z_i^r),f_{\rho_i}(x), f_{\rho_i}(z_i^l))$.
\end{lemma}
\begin{proof}[Proof of Lemma 6.3]
If a lift $\tilde{C}_j$ of $C_j$ divides ideal points $f_{\rho_i}(y),f_{\rho_i}(z_i^r),f_{\rho_i}(x), f_{\rho_i}(z_i^l)$, then $\tilde{C}_j$ intersects $\bar{Q}^l \cup \bar{Q}^r$, where the closure is taken in $\mathbb{H}^2 \bigsqcup \partial \mathbb{H}^2$.
It contradicts that $C_i$ and the projections of $Q^l$ and $Q^r$, which are ideal triangles spiraling to $C_i$, do not intersect to $C_j$ since $C_j$ is a pants-decomposing curve. 
Hence ideal points $f_{\rho_i}(y),f_{\rho_i}(z_i^r),f_{\rho_i}(x), f_{\rho_i}(z_i^l)$ are in a same plaque of the geodesic lamination $\mathscr{C}_j$ which is defined by the preimage of $C_j$.
Since the cross ratio is invariant for isometries, we obtain the statement of the lemma.
\end{proof}
Thus we can deform the original structure $\rho$ of $S$ to a hyperbolic structure $\eta$ by a twist deformation along each $C_i$ so that, for each $i$, $\eta$ realizes the equation  
\[ \log -z(f_{\eta}(y),f_{\eta}(z_i^r),f_{\eta}(x),f_{\eta}(z_i^l))^{-1} = w_i .\]
Then $\Phi_{\mathcal{L}_K} ( \iota_n \circ \eta)$ coincides with the given point  defined by $\tau_{pqr}(T^0_{ijk}) , \tau_{pqr}(T^1_{ijk})= 0, \sigma_p(B_{ij}) = z_{ij}, \theta_p(C_i) = w_i$.
We finish the proof of Theorem 6.1.
\end{proof}

\section{The case of surfaces with boundary}
To define the Bonahon-Dreyer parameterization for surfaces with boundary, Bonahon-Dreyer used the result of Labourie-McShane.
\begin{theorem}[Labourie-McShane \cite{LaMc09} Theorem 9.1.]
Let $S$ be a compact hyperbolic oriented surface with nonempty boundary, and $\rho: \pi(S) \rightarrow {\rm PSL}_n\mathbb{R}$ be a Hitchin representation.
Then there exists a Hitchin representation $\hat{\rho} : \pi_1(\hat{S}) \rightarrow {\rm PSL}_n\mathbb{R}$ of the fundamental group of the double $\hat{S}$ of $S$ such that the restriction $\hat{\rho}$ to $\pi_1(S)$ is equal to $\rho$.
\end{theorem}
For the flag curve $\hat{\xi}_{\hat{\rho}} : \partial \pi_1(\hat{S}) \rightarrow {\rm Flag}(\mathbb{R}^n)$, we set $\xi_{\rho} = \hat{\xi}_{\hat{\rho}}|\partial \pi_1(S)$, the restriction to the boundary of $S$.
We call this restriction the {\it restricted flag curve}.
We use this restriction to define the Bonahon-Dreyer parameterization of surfaces with boundary.
As the case of closed surfaces, we consider triangle, shearing, and gluing invariants defined by restricted flag curves.
In particular, the parameterization map $\Phi_{\mathcal{L}_K} : H_n(S) \rightarrow \mathbb{R}^N$ is defined by
\begin{itemize}
\item all triangle invariants for ideal triangles which give the ideal triangulation by $\mathcal{L}_K$,
\item all shearing invariants for biinfinite leaves of $\mathcal{L}_K$, and
\item all gluing invariants for closed leaves of $\mathcal{L}_K$ which are not a boundary component of $S$.
\end{itemize}

The range is the interior of a convex polytope in $\mathbb{R}^N$.
The convex polytope is defined by the closed equality condition for closed leaves which are not on boundary of $S$, and the closed inequality condition for boundary components.
Here the closed inequality condition is the condition $L_p^{\rho}~ \mbox{or}~R_p^{\rho}(C) > 0$.

In this case, Theorem 5.6 and Theorem 6.1 also hold.
To check this, we focus on the doubling construction of ${\rm PSL}_n\mathbb{R}$-Fuchsian representations.
In the proof of the existence of Hitchin doubles (Theorem 9.1 of \cite{LaMc09}), we can see that the double of a Fuchsian representation $\iota_n \circ \rho$ is $\iota_n \circ \hat{\rho}$, the Fuchsian representation induced by the hyperbolic double $\hat{\rho}$ of the hyperbolic holonomy $\rho$.
Thus the restricted flag curve of $\iota_n \circ \rho$ is the restriction of the Veronese flag curve of $\iota_n \circ \hat{\rho}$ and Theorem 5.6 and Theorem 6.1 for non-closed surface are shown similarly.

\end{document}